\documentclass[10pt, english]{article}

\usepackage[margin= 2 cm]{geometry}

\usepackage{amsthm}
\usepackage{amsmath}
\usepackage{amssymb}
\usepackage{setspace}
\usepackage{mathtools}
\usepackage{verbatim}
\usepackage{csquotes}
\usepackage{graphicx}
\usepackage[hidelinks]{hyperref}
\usepackage{thm-restate}
\usepackage{cleveref}
\usepackage{graphicx}
\usepackage{appendix}
\usepackage[inline]{enumitem}
\usepackage{framed}
\usepackage{subcaption}

\usepackage{caption}

\usepackage{floatrow}
\usepackage[T1]{fontenc}

\usepackage{tikz}
\usepackage{mathdots}
\usepackage{xcolor}
\usepackage{diagbox}
\usepackage{colortbl}
\usepackage[absolute,overlay]{textpos}

\graphicspath{ {./images/} }

\usetikzlibrary{calc}
\usetikzlibrary{decorations.pathreplacing}
\usetikzlibrary{positioning,patterns}
\usetikzlibrary{arrows,shapes,positioning}
\usetikzlibrary{decorations.markings}
\def \vx {circle[radius = .07][fill = black]}

\def \smvx {circle[radius = .1][fill = black]}

\tikzstyle{edge}=[very thick]
\definecolor{bostonuniversityred}{rgb}{0.8, 0.0, 0.0}
\definecolor{arsenic}{rgb}{0.23, 0.27, 0.29}
\tikzstyle{diredge}=[postaction={decorate,decoration={markings,
		mark=at position .95 with {\arrow[scale = 1]{stealth};}}}]

\newcommand{\defPt}[3]{
	\def \pt {(#1, #2)}
	\coordinate [at = \pt, name = #3];
}

\tikzset{
   conn/.pic={
     \defPt{0.2}{-0.5}{q0}
     \defPt{-1}{-1.5}{q5}
    \defPt{1}{1.2}{q1}
    \defPt{1}{2.7}{q6}
    \defPt{1.25}{-1.2}{q2}
    \defPt{2.5}{0.6}{q3}
    \defPt{2.5}{-0.6}{q4}
  
        \draw[line width=1 pt] (q0) -- (q1) -- (q3) -- (q4);
        \draw[line width=1 pt] (q2) -- (q3);
        \draw[line width=1 pt] (q0) -- (q5);
        \draw[line width=1 pt] (q1) -- (q6);
  }
}

\newcommand{\fitellipsis}[3] % first and second node names without parentheses
{\draw []let \p1=(#1), \p2=(#2), \n1={atan2(\y2-\y1,\x2-\x1)}, \n2={veclen(\y2-\y1,\x2-\x1)}
    in ($ (\p1)!0.5!(\p2) $) ellipse [ x radius=\n2/2+0.3cm+#3cm, y radius=#3cm, rotate=\n1];
}

\newcommand{\fitellipsiss}[3] % first and second node names without parentheses
{\draw [fill=white]let \p1=(#1), \p2=(#2), \n1={atan2(\y2-\y1,\x2-\x1)}, \n2={veclen(\y2-\y1,\x2-\x1)}
    in ($ (\p1)!0.5!(\p2) $) ellipse [ x radius=\n2/2+#3cm, y radius=#3cm, rotate=\n1];
}

\newcommand{\fitellipsisss}[3] % first and second node names without parentheses
{\draw []let \p1=(#1), \p2=(#2), \n1={atan2(\y2-\y1,\x2-\x1)}, \n2={veclen(\y2-\y1,\x2-\x1)}
    in ($ (\p1)!0.5!(\p2) $) ellipse [ x radius=\n2/2+#3cm, y radius=#3cm, rotate=\n1];
}

\theoremstyle{plain}

\newtheorem*{thm*}{Theorem}
\newtheorem{thm}{Theorem}
\Crefname{thm}{Theorem}{Theorems}
\newtheorem*{lem*}{Lemma}
\newtheorem{lem}[thm]{Lemma}
\Crefname{lem}{Lemma}{Lemmas}

\newtheorem*{claim*}{Claim}
\newtheorem{claim}{Claim}
\Crefname{claim}{Claim}{Claims}
\Crefname{claim}{Claim}{Claims}

\Crefname{prop}{Proposition}{Propositions}

\Crefname{cor}{Corollary}{Corollaries}

\newtheorem{conj}[thm]{Conjecture}
\Crefname{conj}{Conjecture}{Conjectures}

\Crefname{qn}{Question}{Questions}

\Crefname{obs}{Observation}{Observations}

\newtheorem{ex}[thm]{Example}
\Crefname{ex}{Example}{Examples}

\theoremstyle{definition}
\newtheorem{prob}[thm]{Problem}
\Crefname{prob}{Problem}{Problems}

\Crefname{defn}{Definition}{Definitions}

\newtheorem*{defn*}{Definition}

\theoremstyle{remark}

\renewenvironment{proof}[1][]{\begin{trivlist}
\item[\hspace{\labelsep}{\bf\noindent Proof#1.\/}] }{\qed\end{trivlist}}

\newenvironment{cla_proof}[1][]{\begin{trivlist}
\item[\hspace{\labelsep}{\noindent \emph{Proof#1.}\/}] }{\qed\end{trivlist}}

\newcommand{\ceil}[1]{
    \left\lceil #1 \right\rceil
}
\newcommand{\floor}[1]{
    \left\lfloor #1 \right\rfloor
}

\newcommand{\eps}{\varepsilon}

\expandafter\def\expandafter\normalsize\expandafter{%
    \normalsize
    \setlength\abovedisplayskip{8pt}
    \setlength\belowdisplayskip{8pt}
    \setlength\abovedisplayshortskip{4pt}
    \setlength\belowdisplayshortskip{4pt}
}

\usepackage[square,sort,comma,numbers]{natbib}
 \setlist[itemize]{leftmargin=*}

\DeclareFontFamily{OT1}{pzc}{}
\DeclareFontShape{OT1}{pzc}{m}{it}{<-> s * [1.10] pzcmi7t}{}
\DeclareMathAlphabet{\mathpzc}{OT1}{pzc}{m}{it}

\title{\vspace{-0.8cm} Towards the Erd\H{o}s-Hajnal conjecture for $P_5$-free graphs}

\author{Pablo Blanco\thanks{Department of Mathematics, Rutgers University -- New Brunswick, New Brunswick, USA.} \and
Matija Buci\'c\thanks{School of Mathematics, Institute for Advanced Study and Department of Mathematics, Princeton University, Princeton, USA. Email: \href{mailto:matija.bucic@ias.edu} {\nolinkurl{matija.bucic@ias.edu}}.}
}
 \date{}

\begin{document}

\maketitle

\vspace{-0.5cm}
\begin{abstract}
The Erd\H{o}s-Hajnal conjecture is one of the most classical and well-known problems in extremal and structural combinatorics dating back to 1977. It asserts that in stark contrast to the case of a general $n$-vertex graph, if one imposes even a little bit of structure on the graph, namely by forbidding a fixed graph $H$ as an induced subgraph, instead of only being able to find a polylogarithmic size clique or an independent set, one can find one of polynomial size. Despite being the focus of considerable attention over the years, the conjecture remains open. In this paper, we improve the best known lower bound of $2^{\Omega(\sqrt{\log n})}$ on this question, due to Erd\H{o}s and Hajnal from 1989, in the smallest open case, namely when one forbids a $P_5$, the path on $5$ vertices. Namely, we show that any $P_5$-free $n$-vertex graph contains a clique or an independent set of size at least $2^{\Omega(\log n)^{2/3}}$. We obtain the same improvement for an infinite family of graphs.
\end{abstract} 

\section{Introduction} % unfinished at the moment

We will refer to the set of vertices spanning a clique or an independent set in a graph $G$ as a \textit{homogeneous set} and denote the largest size of such a set in $G$ by $\hom(G)$.
A \textit{pure pair} in a graph is a pair of disjoint vertex sets spanning a complete bipartite graph between them in either $G$ or its complement. We will refer to the former case as a \textit{complete pair} and the latter as an \textit{anti-complete pair}.

The following is one of the most classical open problems in extremal and structural combinatorics, due to Erd\H{o}s and Hajnal \cite{erdos-hajnal-1} from 1977.
\begin{conj}\label{conj:E-H}
For any graph $H$ there exists an $\eps>0$ such that any $n$-vertex $H$-free graph $G$ satisfies $\hom(G) \ge n^{\eps}$.
\end{conj}

The Erd\H{o}s-Hajnal conjecture is open despite considerable attention over the years. It is, however, known to hold for a number of families of graphs. In addition, any new graph for which the conjecture is proved leads to an infinite family of graphs for which it holds via a beautiful result of Alon, Pach and Solymosi \cite{APS01}, which allows one to generate new graphs for which the conjecture is true via the substitution operation (see \Cref{sec:substitution} for more details). With this in mind, as well as with the hope that ideas developed might lead to much more general results or even lead to a full proof of \Cref{conj:E-H}, trying to prove the conjecture for specific small graphs attracted quite a lot of attention over the years, see e.g.\ the survey of Gy\'arf\'as \cite{gyarfas-survey}. In particular, by the results and observations in \cite{erdos-hajnal-2,APS01,gyarfas-survey}, the conjecture holds for all graphs on at most four vertices as well as for all graphs on five vertices apart from three special graphs (or their complements) which were explicitly mentioned in \cite{gyarfas-survey} over 25 years ago. The first of the three graphs, for which the conjecture has since been resolved, by Chudnovsky and Safra \cite{bull}, is the so-called ``bull'' graph. The second is $C_5$, for which the conjecture has been proved in a recent breakthrough paper of Chudnovsky, Scott, Seymour and Spirkl \cite{5-hole} building on an idea of Tomon \cite{tomon}, which was further developed by Pach and Tomon in \cite{pach-tomon}.  Unfortunately, despite hope in the area that solving the $C_5$ case, which was raised explicitly by Erd\H{o}s and Hajnal \cite{erdos-hajnal-2}, would pave the road to the resolution of the full conjecture, this has not yet been the case. The new approach they introduce does lead to a number of interesting, more general results surrounding the conjecture.

Given this, the focus has shifted to the only remaining five vertex graph for which the Erd\H{o}s-Hajnal conjecture is not known, namely the $P_5$, the path on five vertices. This problem has been raised explicitly by Gy\'arf\'as \cite{gyarfas-survey} over 25 years ago and its importance as the smallest open case of Erd\H{o}s-Hajnal conjecture was highlighted in a very recent work of Scott, Seymour and Spirkl \cite{scott2021polynomial}. There has been plenty of work surrounding the problem of excluding $P_5$ in addition to other graphs. For some examples, see \cite{p5-work-1,p5-work-2,p5-work-3, p5-work-4,5-hole} and references within. 

In terms of bounds, the best known bound in the Erd\H{o}s-Hajnal conjecture for $P_5$-free graphs matches the best known general bound, due to Erd\H{o}s and Hajnal \cite{erdos-hajnal-2} with the aforementioned result of Scott, Seymour and Spirkl \cite{scott2021polynomial} giving an alternative proof in the case of $P_5$ as a corollary. In particular, Erd\H{o}s and Hajnal showed in 1989 that for any $H$ any $H$-free graph $G$ satisfies $\hom(G)\ge 2^{\Omega(\sqrt{\log n})}$. In general, prior to the current work, there has been no graphs for which an improvement over this bound is known except the ones for which the full proof of the conjecture is known. The case of $C_5$ was in this class for a short time prior to the above-mentioned complete resolution of this case, thanks to a very nice paper of Chudnovsky, Fox, Scott, Seymour and Spirkl \cite{towards-5-hole} who showed an improved bound of $\hom(G)\ge 2^{\Omega(\sqrt{\log n \log \log n})}$ for any $C_5$-free graph $G$.  

The main result of this paper is a significantly stronger improvement in the $P_5$-free case of the Erd\H{o}s-Hajnal conjecture.

\begin{thm}\label{thm:main}
Any $P_5$-free $n$-vertex graph $G$ satisfies $\hom(G) \ge 2^{\Omega(\log n)^{2/3}}$. 
\end{thm}

Let us highlight the fact that this result breaks a substantial barrier arising in a very natural approach for attacking the Erd\H{o}s-Hajnal conjecture first introduced by Erd\H{o}s and Hajnal \cite{erdos-hajnal-2}. Namely, in order to obtain their classical $2^{\Omega(\sqrt{\log n })}$ bound, they observed it is enough to find an (almost) pure pair with both sides of size at least $n/m^{O(1)},$ where $m=\hom(G)$. In order to improve the bound to $2^{\Omega(\sqrt{\log n \log \log n})}$ it is enough to find an (almost) pure pair with one part of size at least $n/m^{O(1)}$ and the other of size $\Omega(n)$. However, to improve the bound beyond this, one already needs to find an (almost) pure pair with one part of size at least $n/m^{O(1)}$ and the other of size at least $n-o(n)$.

This showcases the steep increase in difficulty involved in improving the classical bound of Erd\H{o}s and Hajnal and establishes $2^{\Omega(\sqrt{\log n \log \log n})}$ as a very natural, highly non-trivial, barrier arising when one follows this approach. We note that the improved bound in the $C_5$ case \cite{towards-5-hole}, which preceded its resolution, precisely matched this barrier and indeed followed the above-mentioned approach. In an upcoming work, the second author, Nguyen, Scott, and Seymour \cite{tung-paper}, manage to match this barrier for all graphs $H$, and hence obtain the first improvement to the classical bound of Erd\H{o}s and Hajnal which applies for all graphs.

It is not hard to extend the aforementioned result of Alon, Pach and Solymosi to show that the family of graphs $H$ for which an $n$-vertex graph $G$ being $H$-free implies $\hom(G) \ge 2^{\Omega(\log n)^{2/3}}$ is closed under substitution (see \Cref{sec:substitution}). This means that \Cref{thm:main} immediately provides us with an infinite family of graphs for which we obtain the same improvement.

Let us also highlight the following intermediate result concerning pure pairs in $P_5$-free graphs, which is the starting point of our proof of \Cref{thm:main}, and which we find to be of independent interest. 
Let us also note that due to their close relation to the Erd\H{o}s-Hajnal conjecture, pure pairs have also been extensively studied (see e.g.\ \cite{noga-pairs,thomasse-pure-pairs,pach-pairs,pure-pair-1,pure-pair-2,pure-pair-3} and references therein) and the results of this form in which we get the same size guarantee on \emph{both} parts of the pair are quite rare.% (see e.g.\ \cite{} for some notable exceptions).

\begin{thm}\label{thm:main-pair}
Let $G$ be a $P_5$-free graph on $n$ vertices with $m:=\hom(G) < n$. Then, we can find an anti-complete pair $(A,B)$ with $|A|,|B| \geq \frac{n}{m^{13}}$. 
\end{thm}

%In fact this result does generalise to a family of graphs which we call $1$-error vertex blow-ups. A graph belongs to this family if it can be obtained from a graph which satisfies the Erd\H{o}s-Hajnal property by duplicating a vertex and joining it to the rest of the graph in the exact same way except for some vertex which we join in the opposite way, the edge between the almost duplicate vertices may be arbitrary. For example $P_5$ may be obtained from $P_4$ by duplicating one of its endvertices with an error and in general one may obtain $P_{i+1}$ from $P_i$ this way.
%This operation (without an error) is a special case of a very powerful substitution operation introduced by Alon, Pach and Solymosi \cite{APS01}.

\textbf{Notation.}
All our logarithms are in base two. By default, whenever we say subgraph, we mean an induced subgraph. In very rare instances in which we refer to not necessarily induced subgraphs, we will refer to them as such. $\mathbb{N}$ for us denotes the set of non-negative integers.

\section{Finding a large anticomplete pair}\label{sec:anticomplete-pair}
In this section, we show how to find a large anticomplete pair of sets in a $P_5$-free graph, namely, we prove \Cref{thm:main-pair}. We note that one can prove this result by using the fact that the Erd\H{o}s-Hajnal Conjecture holds for a five-vertex graph $H_5$ which is the vertex-disjoint union of a path on two and three vertices.
Indeed, by the result of Alon, Pach and Solymosi \cite{APS01}, we have that an $n$-vertex graph without a homogeneous set of size $m$ has at least $n^5/m^{O(1)}$ copies of $H_5$. This means that there are three vertices which extend into at least $n^2/m^{O(1)}$ copies of $H_5$ where the three vertices always play the role of the same three vertices of $H_5$, which include the middle vertex of the path on three vertices. Looking at the ``candidate'' sets for the remaining two vertices, one observes their product needs to have size at least $n^2/m^{O(1)}$; so, in particular, they both need to have size at least  $n/m^{O(1)}$ and one cannot have any edges between them without creating an induced $P_5$, completing the proof. We decided to keep the following somewhat longer proof as it is more self-contained and serves as a gentle introduction for some ideas which will be used in the subsequent section.

We begin with a prerequisite, well-known result (see \cite{gyarfas-survey}), the proof of which we include for completeness. %  well-known \Cref{conj:E-H}
Let us denote the four-vertex graph consisting of a path on three vertices together with a fourth vertex which is isolated by $F_4$. The following lemma is a special case of the Erd\H{o}s-Hajnal conjecture (\Cref{conj:E-H}) for $H=F_4$. 

\begin{lem}\label{lem:F4-free}
    Any $m^3$-vertex $F_4$-free graph $G$ contains a homogeneous set of size at least $m$.
\end{lem}

\begin{proof}
Suppose first towards a contradiction that $\alpha(G), \omega(G) \le m-1$.

Let $v$ be a vertex with the maximum number of non-neighbours in $G$, say $d$. Note that by Tur\'an's theorem applied to $\overline{G}$, (see e.g.\ \cite{alon-spencer}) we know that $\omega(G) = \alpha(\overline{G})\ge \frac{|\overline{G}|}{\Delta(\overline{G})+1} = \frac{m^3}{d+1}$. This implies that $d\ge \frac{m^3}{m-1}-1> m^2.$

Let $M$ be the set of non-neighbours of $v$ in $G$. Now observe that $H=G[M]$ does not have $P_3$ as a subgraph, or we obtain a copy of $F_4$.  This implies that for any $w \in V(H)$, we must have $N_H(w) \cup \{w\}$ being a clique. In particular, we can find a clique of size $\Delta(H)+1$. Once again, by Tur\'an's theorem, we know $\alpha(G) \ge \alpha(H) \ge \frac{|H|}{\Delta(H)+1}$.
Since $d=|H|$ and $\omega(H) \ge \Delta(H)+1$ we obtain $\omega(G)\cdot \alpha(G) \ge \omega(H)\cdot \frac{|H|}{\Delta(H)+1} \ge \omega(H) \cdot \frac{d}{\omega(H)} = d > m^2$, which is a contradiction. 
\end{proof}

A second prerequisite which we will be using repeatedly throughout the paper, is the following easy observation.
\begin{lem}\label{lem:cute-nice-obs}
Let $C$ be a connected component of a graph $G$ with $|C|\ge 2$ and $v\in V(G)\setminus C$. If $(\{v\}, C)$ is not a pure pair, then we can find vertices $u,w\in C$ such that $vw,uw\in E(G)$ and $vu\not\in E(G).$
\end{lem}
\begin{proof}
Suppose there are no such vertices for a contradiction. Since $(\{v\}, C)$ is not a pure pair, we can find vertices $u'$ and $w'$ such that $vu' \not\in E(G)$ and $vw'\in E(G)$. Since $G[C]$ is connected, we can find a path consisting of vertices in $C$ between $u'$ and $w'$. Let $u$ and $w$ be the closest vertices on the path for which $vu \not\in E(G)$ and $vw\in E(G)$. Then there can be no vertex between them because this vertex, together with either $u$ or $w$, would contradict minimality. This shows that $u$ and $w$ must be adjacent, as desired. 
\end{proof}

Before turning to the main lemma of this section, we introduce some convenient notation.
Let $G$ be a graph and $e=xy\in E(\overline{G}) \cup E(G)$ a pair of vertices. 
$$N_{0|1}(e) := \{v\in V(G)\ |\ vx\in E(\overline{G}),\ vy\in E(G) \text{ or } vx\in E(G),\ vy\in E(\overline{G})\};$$
$$N_{00}(e) := \{v\in V(G)\ |\ vx\in E(\overline{G}),\ vy\in E(\overline{G})\}; \ \ \ N_{11}(e) := \{v\in V(G)\ |\ vx\in E(G),\ vy\in E(G)\}.$$

For two disjoint subsets $S,T\subseteq V(G)$, we denote the set of edges and non-edges between $S$ and $T$ by:
$$E(S,T) := \{uv\in E(G)\ |\ u\in S, \ v\in T\}; \quad \quad \quad\quad \overline{E}(S,T) := \{uv\notin E(G)\ |\ u\in S, \ v\in T\}.$$

We are now ready to state and prove the main result of this section.

\begin{thm}\label{lem:main}
Let $G$ be a $P_5$-free graph on $n$ vertices with $m:=\hom(G) < n$. Then, we can find an anti-complete pair $(A,B)$ with $|A|,|B| \geq \frac{n}{m^{13}}$ for which both ${G}[A]$ and ${G}[B]$ are connected. 
\end{thm}
\begin{proof}
Note that since $\hom(G)<n$, we know $G$ is not a clique, so we can find a non-edge. This non-edge provides us with the desired pair unless $n >m^{13} \ge 2^{13}$. On the other hand, a quantitative version of Ramsey's theorem due to Erd\H{o}s and Szekeres \cite{E-S} implies $m=\hom(G) \ge \frac 12 \log n >\frac{13}{2}\log m$ which implies $m \ge 26$. 

By applying \Cref{lem:F4-free} to any subgraph of $G$ on at least $s:=(m+1)^{3}$ vertices, since there is no homogeneous set of size larger than $m$, we conclude this subgraph contains a copy of $F_4$.

\begin{claim} \label{cl:supersaturation}
There are more than $\frac{n^4}{s^4}$ copies of $F_4$ in $G$.
\end{claim}
\begin{cla_proof}
Any vertex subset of size $s$ induces an $F_4$ and there are $\binom{n}{s}$ many of these subsets. Each $F_4$ is counted at most $\binom{n-4}{s-4}$ times so that there are at least $\binom{n}{s}/\binom{n-4}{s-4}$ many distinct copies of $F_4$ in $G$. We have 
$$\frac{\binom{n}{s}}{\binom{n-4}{s-4}}= \frac{n(n-1)(n-2)(n-3)}{s(s-1)(s-2)(s-3)} > \frac{n^4}{s^4},$$
completing the proof.
\end{cla_proof}

Let $M$ be the maximum number of distinct induced $F_4$ subgraphs that share the same edge. Since each copy of $F_4$ contains two edges, we get by double counting and using the previous claim that
$ 2\cdot \frac{n^4}{s^4} \le M \cdot \binom{n}{2}.$
%$$\binom{n}{2}\cdot M \geq |\{(e, H): H \text{ is a copy of }F_4 \text{ and } e\in E(H)\}|\geq \frac{n^4}{s^4}$$
After rearranging, this gives
\begin{equation}\label{cl:btm-max}
M > \frac{4n^2}{s^4}.
\end{equation}

Let us fix an edge $e$ which belongs to $M$ different $F_4$'s. Observe that given any $F_4$ containing $e$, the remaining two vertices need to be one in $N_{00}(e)$ and one in $N_{0|1}(e) $, so the total number of $F_4$'s containing $e$ is precisely the number of non-edges in between these two sets. Rewriting this in terms of an equation, we get
%For an edge $e'\in E(\overline{G})$, let $d_{F_4}(e')$ be the number of distinct copies of $F_4$ which have $e'$ as a bottom edge. 
\begin{equation}\label{eq:density_00_01}
|\overline{E}(N_{0|1}(e),N_{00}(e))| = M
\end{equation}

Henceforth, we abbreviate $N_{ij}:=N_{ij}(e)$ since we only consider such neighbourhoods with respect to the edge $e$.

\begin{claim}\label{cl:consistency10}
    Let $C$ be a connected component of ${G}[N_{00}]$ and $v\in N_{0|1}$. Then either $v$ is adjacent to all vertices in $C$ or $v$ is non-adjacent to all vertices in $C$.
\end{claim}
\begin{cla_proof}
    Suppose to the contrary that $v$ has both a neighbour $u$ and a non-neighbour $w$ among vertices of $C$. Then by \Cref{lem:cute-nice-obs} there exist adjacent vertices $u',w' \in C$ for which $vu' \in E(G)$ and $vw' \notin E(G)$. If we denote vertices of $e$ as $xy$ such that $vx \in E(G)$ and $vy \notin E(G)$, then $w'u'vxy$ gives an induced copy of $P_5$, a contradiction. See \Cref{fig:consistency} for an illustration.
\end{cla_proof}
\begin{figure}
    \centering
    \begin{tikzpicture}[scale=1.1, rotate=-90]
\defPt{0.2}{-0.5}{x0}
\defPt{1}{1.2}{x1}
\defPt{1.25}{-1.2}{x2}
\defPt{2.5}{0.6}{x3}
\defPt{2.5}{-0.6}{x4}

\draw[line width = 1.5 pt] (x3) -- (x4) -- (x1) -- (x2) -- (x0);

\draw[line width = 1 pt, dashed] (x3) -- (x0);
\draw[line width = 1 pt, dashed] (x3) -- (x1);
\draw[line width = 1 pt, dashed] (x3) -- (x2);
\draw[line width = 1 pt, dashed] (x4) -- (x2);
\draw[line width = 1 pt, dashed] (x4) -- (x0);
\draw[line width = 1 pt, dashed] (x1) -- (x0);

\fitellipsis{$(x0)$}{$(x2)$}{0.4};
\fitellipsis{$(x1)+(0.25,0)$}{$(x1)-(0.8,0.7)$}{0.4};

\node[] at ($(x0)+(-0.3,0)$) {$w'$};
\node[] at ($(x1)+(0.3,0.05)$) {$v$};
\node[] at ($(x2)+(0.1,-0.25)$) {$u'$};
\node[] at ($(x3)+(0,0.35)$) {$y$};
\node[] at ($(x4)+(0,-0.35)$) {$x$};
\node[] at ($0.5*(x2)+0.5*(x0)+(-0.4,-0.9)$) {$N_{00}$};
\node[] at ($(x1)+(-0.6,0.5)$) {$N_{0\mid 1}$};

\foreach \i in {0,...,4}
{
\draw[] (x\i) \smvx;
}
\end{tikzpicture}
    \caption{Illustration of a $P_5$ leading to the contradiction in \Cref{cl:consistency10}.}
    \label{fig:consistency}
\end{figure}
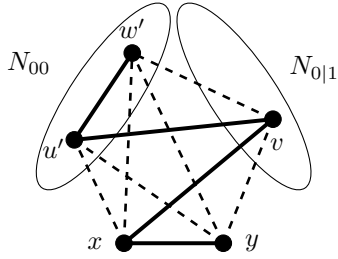

Let $C_1,\ldots ,C_k$ be the connected components of ${G}[N_{00}]$. The above claim guarantees that any $v\in N_{0|1}$ is complete or anti-complete to $C_i$. Since there are no edges between different connected components, we can obtain an independent set by taking one vertex from each. This, in particular, implies that $k \le m$.

We call a $C_i$ \textit{small} if 
$$|C_i| \leq \frac{2n}{s^4m}.$$
Note that this implies that 
$$\sum_{i: \: C_i\text{ is small}} |\overline{E}(C_i,N_{0|1})| \leq \sum_{i: \: C_i\text{ is small}} |C_i|n \leq k \cdot \frac{2n}{s^4m} \cdot n \le  \frac{2n^2}{s^4} < \frac{M}{2}.$$
where in the penultimate inequality, we used $k \le m$.

We say a component $C_i$ has \textit{low degree} if
$$|\overline{E}(C_i,N_{0|1})| \leq \frac{M|C_i|}{2n}.$$ 
Note that this implies 
$$\sum_{i: \: C_i\text{ has low degree}} |\overline{E}(C_i,N_{0|1})| \leq \sum_{i: \: C_i\text{ has low degree}} \frac{M|C_i|}{2n} \le \frac{M}{2n} \cdot \sum_{i=1}^k|C_j|\leq \frac{M}{2}.$$

By \eqref{eq:density_00_01} and since $C_1,\ldots, C_k$ partition the vertices of $N_{00}$, we have
$$ \sum_{i=1}^k |\overline{E}(C_i,N_{0|1})|=|\overline{E}(N_{00},N_{0|1})| = M. $$

Since the contribution of small and low-degree components is less than $M$, we conclude there is some $C_i$ which is neither small nor has low degree. That is,
$$|C_i| > \frac{2n}{s^4m} = \frac{2n}{m(m+1)^{12}}\ge \frac{n}{m^{13}}$$
using that $2 \ge (1+1/m)^{12}$ which holds since $m \ge {26}$ and
$$\frac{|\overline{E}(C_i,N_{0|1})|}{|C_i|} > \frac{M}{2n} \ge \frac{2n}{s^4} = \frac{n}{m^{12}}.$$ 
The second inequality guarantees that there are at least $\frac{n}{m^{12}}$ vertices in $N_{0|1}$ which have a non-neighbour in $C_i$ and, therefore by \Cref{cl:consistency10}, are anti-complete to $C_i$. Let the collection of these vertices be $B'$. There are at most $m$ connected components in ${G}[B']$, so there is some component $B\subseteq B'$ that has $|B| \geq \frac{n}{m^{13}}$. By construction, we may let $A:=C_i$ to find the desired pair of subsets. 
\end{proof} 

\section{Finding a large homogeneous set in $P_5$-free graphs}
In this section, we prove our main result \Cref{thm:main}. We begin with a few general preliminaries.

\subsection{Preliminaries}
It will be convenient for us to work with the following equivalent formulation of the Erd\H{o}s-Hajnal conjecture in terms of cographs. A graph is said to be a \emph{cograph} if it is a single-vertex graph or if it can be obtained by taking a vertex-disjoint union of two smaller cographs in such a way that their vertex sets make a pure pair. Since cographs are known to be perfect \cite{seinsche}, the following conjecture is equivalent to the Erd\H{o}s-Hajnal conjecture.

\begin{conj}\label{conj:E-H-2}
Given any graph $H$, there exists an $\eps>0$ such that any $n$-vertex $H$-free graph $G$ contains a cograph of order $n^{\eps}$ as an induced subgraph.
\end{conj}

The following few technical lemmas concerning concave functions will be helpful in controlling the sizes of certain restricted subsets in our argument. Let us note that we will always apply them to the function $f(x)=2^{c(\log x)^{2/3}}$ but are writing them in a somewhat more general form in the hope that they may aid in proving results with stronger functions. 
\begin{lem}\label{lem:concave}
Let $f:[1,\infty) \to \mathbb{R}$ be a continuous function, twice differentiable on $(1,\infty)$ with $f''(t) \leq 0$ for all $t > 1$. If $y \ge x \ge a+1\ge 1$, then $$f(x)+f(y) \geq f(x-a) + f(y+a).$$
\end{lem}
\begin{proof}
We begin by noting that the second derivative $f''(t) \leq 0$ for $t> 1$ implies the first derivative $f'(t)$ is non-increasing on the interval $[1,\infty)$; so, $f'(x-a+s) \geq f'(y+s)$ for every $s\in [0,a]$ because $x-a+s < y+s$. Then,
$$\int_{x-a}^x f'(t) \,dt\geq \int_y^{y+a} f'(t) \,dt. $$
By the Fundamental Theorem of Calculus, this gives
$$f(x) - f(x-a) \geq f(y+a)-f(y)$$
which gives the desired inequality after rearranging.
%$$f(x) + f(y) \geq f(y+1) + f(x-1)$$
\end{proof} 

The next lemma is also analytic, although it concerns real sequences $f:\mathbb{N} \to \mathbb{R}^+$. Such a sequence is said to be concave if $f(i)-f(i-1)\ge f(i+1)-f(i)$ for all integers $i \ge 1$. An easy summation shows that such sequences satisfy the conclusion of the above easy lemma as well, for any integers $y \ge x \ge a.$
\begin{lem}\label{r-partite-bound}
Let $f:\mathbb{N} \to \mathbb{R}^+$ be concave and increasing function 
and $G$ a complete or anticomplete $r$-partite $s$-vertex graph with each part of size at most $a \in \mathbb{N}$. Assuming that inside of any vertex subset of size $s' \le s$, we can find a cograph of order at least $f(s')$, then in $G$ we can find a cograph of order $\floor{s/a}\cdot f(a)$.
\end{lem}
\begin{proof}
Let $1\le s_1,\ldots, s_r \le a$ denote the sizes of the parts so that $s=s_1+\ldots+s_r$. By applying the assumption to each part $i$, we can find a cograph of size $f(s_i)$ inside of part $i$. Now, since the graph is complete or anticomplete $r$-partite, the union of these cographs is also a cograph of order at least 
$$f(s_1)+\ldots+f(s_r) \ge \underbrace{f(a)+\ldots+f(a)}_{\floor{s/a}}+f\left(s-\floor{s/a}\cdot a\right)\ge\floor{s/a}\cdot f(a),$$
where the first inequality follows by repeated application of the definition of concavity.
\end{proof}

In the previous lemma, we know that all parts are complete or anticomplete to each other. The following lemma is similar to the previous one, but we only know that any two parts make a pure pair.

\begin{lem}\label{r-partite-incons-bound}
Let $f:[1,\infty) \to \mathbb{R}^+$ be a function such that $\log (f(2^x))$ is concave and increasing. Let $G$ be an $s$-vertex graph with vertices partitioned into disjoint parts $V_1,\ldots, V_r$, each of size in $[a,b]$, and such that for each $i<j$, the pair $(V_i,V_j)$ is pure. Assuming that inside of any vertex subset of size $s' \le s$ we can find a cograph of order $f(s')$, then in $G$ we can find a cograph of order at least $\min\left\{f\left(\frac{s}{2a \log (4b/a)}\right)\cdot f(a), f\left( \frac{s}{2b \log (4b/a)}\right)\cdot f(b)\right\}$.
\end{lem}
\begin{proof}
Let us denote for all $\floor{\log a}+1\le i \le \ceil{\log b}$ by $B_i$ the \emph{bucket} $i$ consisting of all $j$ for which $|V_j| \in [2^{i-1},2^i]$. Since there are at most $\ceil{\log b}-\floor{\log a}\le  \log (4{b}/{a})$ buckets, we know that for one of them, say $B_i$, we have $\sum_{j \in B_i} |V_j| \ge \frac{s}{\log (4b/a)}$. Let $t=\min_{j \in B_i} \{|V_j|\}$ so that $a\le t \le b$. Since for each $j \in B_i$ we have $2^{i-1}\le |V_j|\le 2^i$, we know $t\le |V_j|\le 2t$ which implies $|B_i| \ge \frac{s}{2t \log (4b/a)}$. If we apply the assumption to a subgraph of $G$ consisting of precisely one representative vertex from each $V_j$ for $j \in B_i$, we can find inside of it a cograph of size at least $f(|B_i|)$. If we now apply the assumption to each $V_j$ containing a representative vertex of this cograph, then we obtain inside of each $V_j$ a cograph of size at least $f(t).$ The union of these cographs is also a cograph of order at least 
$$f(|B_i|)\cdot f(t) \ge f\left( \frac{s}{2t \log (4b/a)}\right)\cdot f(t).$$
Now we claim that $f(x)f(y)$ is minimised given $xy$ is fixed over all $y \in [a,b]$ at the boundary; that is when $y=a$ or $y=b$. To see this consider the function $g(x)=\log f(2^x)$ and note that our minimisation problem is equivalent to minimising $\log (f(x)f(y))=\log f(x) + \log f(y)$ subject to $xy$ being constant. Substituting $x'=\log x$ and $y'= \log y$ we wish to optimise $g(x')+g(y')$ subject to $2^{x'}2^{y'}$ being fixed. Since this is equivalent to $x'+y'$ being fixed and we assumed $g$ is concave, the claim follows.
\end{proof}

The following simple lemma will also be convenient and used repeatedly in the proof.

\begin{lem}\label{lem:one-third-pure-pair} 
Suppose in an $n$-vertex graph $G$ the size of a maximum connected component is at most $n/3.$ Then there is an anti-complete pair with both sides of size at least $n/3$.
\end{lem}
\begin{proof}
Let $C_1,\ldots, C_t$ be connected components of $G$ and let $i$ be minimal such that $|C_1|+\ldots + |C_{i}| \ge n/3$. Since minimality implies $|C_1|+\ldots + |C_{i-1}| < n/3$ and since by assumption $|C_i| \le n/3,$ we have $|C_{i+1}|+\ldots +|C_t|\ge n/3$. Since there are no edges between distinct components of a graph, we get that $(C_1 \cup \ldots \cup C_i, C_{i+1} \cup \ldots \cup C_t)$ is an anti-complete pair with the desired properties.
\end{proof}

\subsection{Proof of our main result}

In the proof of the following theorem, we will occasionally switch between the graph theoretic and Ramsey theoretic perspectives. In particular, we will refer to edges of a graph as red edges and non-edges as blue edges. Among other things, this will be helpful when illustrating various parts of our arguments. Given two disjoint subsets of vertices $S$ and $T$, we will say $S$ is \textit{red/blue} to $T$ if all edges between $S$ and $T$ are red/blue. So a red pair is simply a complete pure pair, and a blue pair is an anti-complete pure pair.
In case we know that for every vertex in $S$, all its edges towards $T$ are monochromatic, then we say $S$ is \emph{consistent} to $T$.

We restate our main result in the cograph setting and in a more quantitative form.

\begin{thm}\label{thm:main2}
Any $P_5$-free $n$-vertex graph contains a cograph of order at least $2^{(\log n)^{2/3}/16}.$ 
\end{thm}

%\Cref{thm:main}

\begin{proof}
Let $c=\frac{1}{16}$ . 
Let $f(x):=2^{c(\log x)^{2/3}}$ and observe that it is increasing and concave on $[1,\infty)$ so we may apply \Cref{lem:concave} to $f$. Define $f(0)=0$ and observe that the restriction of $f$ on $\mathbb{N}$ is also concave so that we may apply \Cref{r-partite-bound} to $f$ as well. Let us also define for convenience $\eps_n:=\frac{1}{(\log n)^{1/3}}$ so that $f(n)=n^{c\eps_n}$.

Suppose towards a contradiction that there exists a $P_5$-free $n$-vertex graph which does not contain a cograph with $f(n)=2^{c(\log n)^{2/3}}=n^{c\eps_n}$ vertices. Let $G$ be such a graph with the minimum number of vertices. In particular, this implies that 
\vspace{-0.2cm}
\begin{equation}\label{eq:crit-G}
\text{any } m \text{-vertex } H \subsetneq G\text{ contains a cograph of size at least }f(m)=2^{c(\log m)^{2/3}}=m^{c\eps_m}\ge m^{c\eps_n}.
\end{equation}
We note that we will, in most cases, use the above inequality in its most wasteful form $f(m)\ge m^{c\eps_n}$, which is not going to be a big loss for us since we will mostly apply it to very large sets $m$ so that $\eps_m \approx \eps_n$. There is only one place in the proof where we need the more precise form, namely at the very end of the proof when we apply it to much smaller subgraphs.

The result is trivial if $n=1$ and note that any graph with $n \ge 2$ vertices has a homogeneous set of size at least two, so we are done unless $2^{c(\log n)^{2/3}}>2$. This  implies $(\log n)^{1/3} > 4$ so $\eps_n \le 1/4$. %By quantitative Ramsey's theorem due to Erd\H{o}s and Szekeres \cite{E-S} this guarantees a homogeneous set of size at least $\frac 12 \log n \ge 32$ so we may assume $2^{c(\log n)^{2/3}}>32$ and more specifically $c(\log n)^{2/3}>5.$

The main goal of the first part of the proof is to show the following claim.

\textbf{Main claim.} Let $G' \subseteq G$ satisfy $|G'| \ge n-n^{1-\eps_n^2}$. Then we can find $A=A(G'),Blu=Blu(G'), Red=Red(G')$ and
$Err=Err(G')$ which partition $V(G')$ and satisfy the following properties:
\begin{enumerate}[label=\textbf{\alph*)}]
    \item\label{itm:a} $|A|\ge n^{1-\eps_n}$,
    \item\label{itm:b} $|Blu| \ge 2n^{1-\eps_n^2}$ and $(A,Blu)$ is a blue pair,
    \item\label{itm:c} $|Red| \ge 2n^{1-\eps_n^2}$ and $(A,Red)$ is a red pair and every vertex in $Red$ sends at most $n^{\eps_n^2}|A|$ blue edges to $Blu$,
    \item \label{itm:d} $|Err| \le |A|.$
\end{enumerate}

One should think of $A$ as the set of nice new vertices which behaves extremely nicely with respect to the rest of our subgraph $G'$, namely, apart from a tiny error set $Err$, every vertex outside $A$ is consistent to it, and there are relatively many vertices which are red to $A$ and which are blue to $A$.
The role of working in a subgraph $G'$ of almost the full size is so that we can iteratively apply the main claim by repeatedly removing sets $A$ and $Err$ until they cover at least $n^{1-\eps_n^2}$ vertices and, in particular, by the assumption \ref{itm:d} this will guarantee that the $A$ sets have at least this size divided by two. The final part of the argument will be to show any pair of such $A$ sets need to make a pure pair (not all necessarily of the same colour) and then use this to obtain a cograph of desired size via \Cref{r-partite-incons-bound}.

Since the expressions involved are somewhat technical, let us try to give an intuition of how one can think of various subset sizes we will work with. Namely, we think of any set of size smaller than $n^{1-\Theta(\eps_n)}$ as being ``small'', any set of size larger than this is considered to be of at least ``medium'' size, while any set larger than $n^{1-\Theta(\eps_n^2)}$ is considered to be ``large''. Finally, any set of size at least $n-n^{\Theta(\eps_n^2)}$ is considered ``huge''. 

The next two lemmas will establish the so-called ``win'' conditions, namely the sizes of parts of a pure pair which guarantee us a large enough cograph to give us a contradiction. The first one says that we win if we can find a pure pair with both parts large, while the second one, in some sense, says we win if one part of the pair is of size at least medium and the other is huge.

\begin{claim}\label{cl:win-purepair-1}%Needs $c < \frac{1}{3}$ and $n \ge 37
There is no pure pair $(Y, Z)$ in $G$ with $|Y|,|Z| \ge n^{1-\eps_n^2/c}$.
\end{claim}

\begin{cla_proof} % note to self: cliques in blue, indep. sets in red
Let us suppose towards a contradiction that such a pair exists. Using the criticality assumption \eqref{eq:crit-G}, we can find inside of $G[Y]$ a cograph of size at least $f(|Y|)$ and inside of $G[Z]$ a cograph of size $f(|Z|)$. Since $(Y, Z)$ is a pure pair, we can combine these two cographs to get one of order at least $f(|Y|)+f(|Z|)$ in $G$ and using our initial assumption that there is no cograph of order at least $f(n)$ in $G$ this leads to the following contradiction: 
$$f(n) > f(|Y|)+f(|Z|) \geq |Y|^{c\eps_n}+|Z|^{c\eps_n}\ge  2\cdot n^{c\eps_n-\eps_n^3}=n^{c\eps_n}=f(n),$$
where the penultimate equality follows by the definition of $\eps_n$ which implies $n^{\eps_n^3}=2$.
\end{cla_proof}

The following claim is an asymmetric version of the one above, and one can think of it as saying that the existence of a pure pair in which one side is not tiny and one is huge leads to a contradiction, but we do need slightly more flexible notions of tiny and huge here. Although the idea is the same as used above, the proof is slightly technical.

\begin{claim}\label{cl:win-purepair-2}%needs $c < \frac{1}{34}$
There is no pure pair $(Y,Z)$ with both $|Y|,|Z|\ge n^{1-\eps_n/c}$ and $|Y|+|Z|\ge n-7n^{1-\eps_n^2}.$
\end{claim}

\begin{cla_proof} 
Let us suppose towards a contradiction that such a pair exists and suppose without loss of generality that $|Y| \le |Z|$. Let us denote by $x:=|Y|-n^{1-\eps_n/c}$. 
Then, $|Z|+x\ge n-7n^{1-\eps_n^2}-n^{1-\eps_n/c}\ge n-8n^{1-\eps_n^2}$. 
% c\eps_n\le 1

Using the criticality of $G$ via \eqref{eq:crit-G} implies we can find inside $G[Y]$ and $G[Z]$ a cograph of order $f(|Y|)$ and $f(|Z|)$ respectively. Since $(Y, Z)$ is a pure pair, we can combine them into a cograph of order at least $f(|Y|)+f(|Z|)$ in $G$. Using our initial assumption that there is no cograph of order at least $f(n)$ in $G$ together with the concavity of $f(x)$ through \Cref{lem:concave} as well as the fact it is increasing we get the following contradiction
\begin{align*}
  f(n) > f(|Y|)+f(|Z|) \ge f(|Y|-x) +f(|Z|+x) &\ge f\left(n^{1-\eps_n/c}\right)+f\left(n-8n^{1-\eps_n^2}\right)\\
  &=  n^{c\eps_n-\eps_n^2}+\left(n-8n^{1-\eps_n^2}\right)^{c\eps_n}\\
  &= n^{c\eps_n-\eps_n^2}+n^{c\eps_n}\cdot\left(1-8n^{-\eps_n^2}\right)^{c\eps_n}\\
  &\ge n^{c\eps_n-\eps_n^2}+n^{c\eps_n}\cdot\left(1-n^{-\eps_n^2}\right)=n^{c\eps_n}=f(n),
\end{align*}
where in the last inequality we used $c\eps_n \le \frac1{16}$ and $n^{-\eps_n^2}=2^{-(\log n)^{1/3}}\le \frac1{16}$ so that we may apply the inequality $(1-8y)^{1/16}\ge 1-y$ which holds for all $y\le \frac1{16}$.
\end{cla_proof}

% \begin{claim}\label{cl:win-purepair-3}
% Given \eqref{eq:crit-G} and $c < \frac{1}{34}$
% , there is no pure pair $(Z,Y)$ with $|Z|\ge \frac{n}{2^{36c(\log n)^{2/3}}}$ and $|Y|\ge n-\frac{2n}{2^{(\log n)^{1/3}}}.$ 
% \end{claim}

The above lemmas were the final preliminary steps of the proof. We are now ready to start with the main part of the proof of the main claim.
The first step is to find an anti-complete pure pair with both parts being not tiny, and this will come courtesy of \Cref{lem:main}.
We say that a pair of disjoint sets of vertices $(A, B)$ is \emph{good} if all edges between $A$ and $B$ are blue, both $A$ and $B$ span connected subgraphs in red and $|A|\le |B|$. In particular, since a homogeneous set is a cograph, we have $\hom(G)<f(n)$ and \Cref{lem:main} guarantees us that we can find a good pair $(A, B)$ in $G'$ such that 
\begin{equation}\label{eq:A-size-bound}
|B| \ge |A| \geq \frac{n}{(f(n))^{13}}= n^{1-13c\eps_n} \ge n^{1-\eps_n}.
\end{equation}
We pick $(A, B)$ to be a maximum size good pair, in the lexicographic sense, so for any other good pair $(A', B')$, we have either $|A'| < |A|$ or if $|A'| = |A|$, then $|B'| \leq |B|$. In particular, \eqref{eq:A-size-bound} holds for our pair. We note also that the assumption $|A|\le |B|$ only plays a role much later in the proof, so up until that point, we may and will treat $A$ and $B$ as symmetric in order to avoid repeating the same argument twice. We will clearly mark the point in the proof at which we break the symmetry.

We begin with a simple claim which already imposes quite a strong structural restriction on our colouring.

\begin{claim}\label{cl:consistent-AB} 
Every vertex $v\not\in A\cup B$ is consistent to $A$ or consistent to $B$.
\end{claim}

\begin{cla_proof}
Suppose otherwise. Then, we have some vertex $v$ which is not consistent with respect to both $A$ and $B$. This means via \Cref{lem:cute-nice-obs} that there are vertices $a,a' \in A$ such that $va,aa'$ are red and $va'$ is blue and $b,b' \in B$ such that $vb,bb'$ are red and $vb'$ is blue. Since $(A,B)$ is a blue pair, all edges between $\{a,a'\}$ and $\{b,b'\}$ are blue which means $a'avbb'$ induces a red $P_5$, a contradiction. See \Cref{fig:consistent-AB} for an illustration.
\end{cla_proof}
\begin{figure}
\RawFloats
\begin{minipage}[t]{0.49\textwidth}
\centering
\captionsetup{width=\textwidth}
\begin{tikzpicture}[scale=1.1, rotate=90]
\defPt{0}{0}{x0}
\defPt{1.5}{1.2}{x1}
\defPt{1.5}{-1.2}{x2}
\defPt{2.5}{0.6}{x3}
\defPt{2.5}{-0.6}{x4}

\fitellipsiss{$(x1)$}{$(x3)$}{0.6};
\fitellipsiss{$(x2)$}{$(x4)$}{0.6};

\draw[line width = 2 pt, red]  (x3) -- (x1) -- (x0) -- (x2) -- (x4);

\draw[line width = 1 pt, blue] (x3) -- (x2);
\draw[line width = 1 pt, blue] (x3) -- (x4);
\draw[line width = 1 pt, blue] (x3) -- (x0);
\draw[line width = 1 pt, blue] (x4) -- (x0);
\draw[line width = 1 pt, blue] (x4) -- (x1);
\draw[line width = 1 pt, blue] (x2) -- (x1);

\node[] at ($(x0)+(-0.3,0)$) {$v$};
\node[] at ($(x1)+(-0.1,0.25)$) {$a$};
\node[] at ($(x2)+(-0.1,-0.25)$) {$b$};
\node[] at ($(x3)+(0.1,0.25)$) {$a'$};
\node[] at ($(x4)+(0.1,-0.25)$) {$b'$};
\node[] at ($0.5*(x1)+0.5*(x3)+(0.4,0.9)$) {$A$};
\node[] at ($0.5*(x2)+0.5*(x4)-(-0.4,0.9)$) {$B$};

\foreach \i in {0,...,4}
{
\draw[] (x\i) \smvx;
}
\end{tikzpicture}
    \captionsetup{width=\textwidth}
    \caption{$P_5$ leading to the contradiction in \Cref{cl:consistent-AB}.}
    \label{fig:consistent-AB}
\end{minipage}\hfill
\begin{minipage}[t]{0.49\textwidth}
\centering
\begin{tikzpicture}[yscale=0.8]
\defPt{0}{0}{D}
\defPt{-1}{0}{C}
\defPt{1}{0}{E}
\defPt{-1.5}{2}{A}
\defPt{1.5}{2}{B}
\defPt{0}{4}{X}

%\draw[line width = 1.5 pt] (x3) -- (x4) -- (x1) -- (x2) -- (x0);

\foreach \i in {-2,-1,0,1,2}
{
    \foreach \j in {-1,0,1}
    {
    \draw[line width = 1 pt,blue] ($(X)+0.3*(\i,0)$) -- ($(A)+0.2*(\j,0)$);
    }
}

\foreach \i in {-2,-1,0,1,2}
{
    \foreach \j in {-1,0,1}
    {
    \draw[line width = 1 pt,blue] ($(X)+0.3*(\i,0)$) -- ($(B)+0.2*(\j,0)$);
    }
}

\foreach \i in {-2,-1,0,1,2}
{
    \foreach \j in {-2,-1,0,1,2}
    {
    \draw[line width = 1 pt,blue] ($(A)+0.3*(0,\i)$) -- ($(B)+0.3*(0,\j)$);
    }
}

\foreach \i in {-2,-1,0,1,1}
{
    \foreach \j in {-1,0,1}
    {
    \draw[line width = 1 pt,red] ($(A)+0.2*(0,\i)$) -- ($(D)+(0.5,0)+0.2*(\j,0)$);
    }
}

\foreach \i in {-2,-1,0,1,1}
{
    \foreach \j in {-1,0,1}
    {
    \draw[line width = 1 pt,red] ($(B)+0.2*(0,\i)$) -- ($(D)+(-0.5,0)+0.2*(\j,0)$);
    }
}

\foreach \i in {-2,-1,0,1,1}
{
    \foreach \j in {-2,-1,0,1,2}
    {
    \draw[line width = 1 pt,red] ($(A)+0.3*(0,\i)$) -- ($(C)+0.15*(\j,0)$);
    }
}

\foreach \i in {-2,-1,0,1,1}
{
    \foreach \j in {-2,-1,0,1,2}
    {
    \draw[line width = 1 pt,red] ($(B)+0.3*(0,\i)$) -- ($(E)+0.15*(\j,0)$);
    }
}

\fitellipsiss{$(C)$}{$(E)$}{0.6};
\fitellipsiss{$(A)+(0.2,0.5)$}{$(A)-(0.2,0.5)$}{0.6};
\fitellipsiss{$(B)+(-0.2,0.5)$}{$(B)-(-0.2,0.5)$}{0.6};
\fitellipsiss{$(X)+(0.5,0)$}{$(X)-(0.5,0)$}{0.6};

\draw[] ($(X)+(0.6,-0.5)$) -- ($(X)+(0.6,0.5)$);
\draw[] ($(X)+(0,-0.6)$) -- ($(X)+(0,0.6)$);
\draw[] ($(X)+(-0.6,-0.5)$) -- ($(X)+(-0.6,0.5)$);

\draw[] ($0.5*(C)+0.5*(D)+(0,-0.57)$) -- ($0.5*(C)+0.5*(D)+(0,0.57)$);
\draw[] ($0.5*(E)+0.5*(D)+(0,-0.57)$) -- ($0.5*(E)+0.5*(D)+(0,0.57)$);

\pic[scale=0.25,rotate=15,red] at ($(A)-(0.2,0.15)$) {conn};
\pic[scale=0.25,rotate=35,red] at ($(B)-(0.2,0.15)$) {conn};
\pic[scale=0.1,red] at ($(X)-(0.9,0)$) {conn};
\pic[xscale=0.13, yscale=0.2,red] at ($(X)-(0.4,0.1)$) {conn};
\pic[scale=0.1,red] at ($(X)+(0.73,0)$) {conn};
\pic[xscale=0.13,yscale=0.2, red] at ($(X)+(0.2,-0.1)$) {conn};

\node[] at ($(X)+(-1.5,0)$) {$X$};
\node[] at ($(A)+(-0.9,0.2)$) {$A$};
\node[] at ($(B)+(0.9,0.2)$) {$B$};
\node[] at ($(C)$) {$C$};
\node[] at ($(D)$) {$D$};
\node[] at ($(E)$) {$E$};
% \node[] at ($(x4)+(0,-0.35)$) {$x$};
% \node[] at ($0.5*(x2)+0.5*(x0)+(-0.4,-0.9)$) {$N_{00}$};
% \node[] at ($(x1)+(-0.6,0.5)$) {$N_{0\mid 1}$};

\end{tikzpicture}
    \captionsetup{width=\textwidth}
    \caption{Overall structure of the colouring.}
    \label{fig:overall}
\end{minipage}
\end{figure}

\begin{claim}\label{cl:X-claim}
Any vertex $v\not\in A\cup B$ which is blue to $A$ or blue to $B$ is blue to $A\cup B$. 
\end{claim}
\begin{cla_proof}
Assume for the sake of contradiction that some vertex $v$ is blue to $A$, but not $B$. Then, $B \cup \{v\}$ induces a connected graph in red which is complete to $A$ in blue. This means that $(A, B \cup \{v\})$ forms a good pair, which violates the maximality of $(A, B)$.  We can repeat the argument with roles of $A$ and $B$ switched to obtain the symmetric case. 
\end{cla_proof}

By \Cref{cl:consistent-AB} and \Cref{cl:X-claim}, we can partition (allowing sets to be empty) the vertices of $V(G)\setminus (A\cup B)$ so that: $X$ are the vertices blue to $A\cup B$; $C$ are red to $A$, but not consistent to $B$; $D$ is red to $A\cup B$; and $E$ is red to $B$ but not consistent to $A$. See \Cref{fig:overall} for an illustration.

We now show that $X$ is not much larger than $A$.
\begin{claim}\label{cl:bound-on-X}
$$|X| \le |A|\cdot 2n^{c\eps_n^2}.$$
\end{claim}

\begin{cla_proof}
Let us assume towards a contradiction that 
$$\frac{|X|}{|A|} > 2n^{c\eps_n^2} \ge 1.$$
Let us partition $G[X]$ into connected components $X_1,\ldots, X_r$ in red. Observe that $|X_i| \le |A|$ since otherwise we could use the pair $(X_i,B)$ to contradict maximality of $(A,B)$. Observe also that $G[X]=G[X_1 \cup \ldots \cup X_r]$ is an $r$-partite complete graph in blue with each part of size at most $|A|$. So \Cref{r-partite-bound}, applied with our concave, increasing function $f(n)=2^{c(\log n)^{2/3}}$, to this graph guarantees a cograph of size at least $\floor{|X|/|A|}f(|A|)$ which is by our initial assumption smaller than $f(n)$. Plugging in the values, we obtain the following contradiction
\begin{align*}
    n^{c\eps_n}>\floor{\frac{|X|}{|A|}}\cdot |A|^{c\eps_n}
                       \ge n^{c\eps_n^2} \cdot \left(n^{1-\eps_n}\right)^{c\eps_n}
                       = n^{c\eps_n},
\end{align*}
where in the first inequality we used $\floor{x}\ge x/2$ for $x \ge 1$ and in the second we used \eqref{eq:A-size-bound}.\end{cla_proof}

For certain structural reasons, which will become apparent soon, we would now like $G[A]$ and $G[B]$ to be connected in blue. Unfortunately, this might not be the case and we are forced to introduce $A'$ and $B'$ as the largest blue components of $G[A]$ and $G[B]$, respectively. While we are for a while going to focus on developing a similar, and in fact more detailed picture with respect to the blue pair $(A', B')$ in place of $(A, B)$, we do lose the maximality assumption, which is why we will keep in mind (and in our figures) the maximal pair $(A, B)$ and gradually conclude the structure with respect to it from the one with respect to $(A', B')$.

We note that since there can be no red clique of size greater than $f(n)$, there are at most $f(n)$ connected components in the blue graph of both $G[A]$ and $G[B]$. So since we picked $A'$ and $B'$ to span the largest component we get from \eqref{eq:A-size-bound}:
\begin{equation}\label{eq:a-prime-bnd}
    |A'| \geq \frac{|A|}{f(n)} \geq n^{1-\eps_n},\quad\quad\quad
    |B'| \geq \frac{|B|}{f(n)} \geq n^{1-\eps_n}
\end{equation}
This shows $A', B'$ are both still of at least medium size and will allow us to use \Cref{cl:win-purepair-2} to establish they are not in a pure pair with the other side being huge. 

Let $D'\supseteq D$ be the set of vertices red to both $A'$ and $B'$. Define $C'\subseteq C$ to be the set of vertices in $C$ which are not red to $B'$ and similarly $E'\subseteq E$ as the set of vertices of $E$ not red to $A'$. Note that by construction $C' \cup D' \cup E' = C \cup D \cup E$. We note that the individual sets $C, D$, and $E$ will not play a role in the argument after this point. See \Cref{fig:overall3} for an illustration of the structure established by this point.

\begin{figure}
\RawFloats
\begin{minipage}[t]{0.495\textwidth}
\centering
\captionsetup{width=\textwidth}
    \begin{tikzpicture}[yscale=0.8]
\defPt{0}{0}{D}
\defPt{-2}{0}{C}
\defPt{2}{0}{E}
\defPt{-1.5}{2}{A}
\defPt{1.5}{2}{B}
\defPt{0}{4}{X}

\foreach \i in {-2,-1,0,1,2}
{
    \foreach \j in {-1,0,1}
    {
    \draw[line width = 1 pt,blue] ($(X)+0.3*(\i,0)$) -- ($(A)+0.2*(\j,0)$);
    }
}

\foreach \i in {-2,-1,0,1,2}
{
    \foreach \j in {-1,0,1}
    {
    \draw[line width = 1 pt,blue] ($(X)+0.3*(\i,0)$) -- ($(B)+0.2*(\j,0)$);
    }
}

\foreach \i in {-2,-1,0,1,2}
{
    \foreach \j in {-2,-1,0,1,2}
    {
    \draw[line width = 1 pt,blue] ($(A)+0.3*(0,\i)$) -- ($(B)+0.3*(0,\j)$);
    }
}

\foreach \i in {-3,...,3}
{
    \foreach \j in {-2,-1,0,1,2,3,4,5,6,7}
    {
    \draw[line width = 0.5 pt,red] ($(A)+0.3*(0,\i)$) -- ($(C)+0.15*(\j,0)$);
    }
}

\foreach \i in {-3,...,3}
{
    \foreach \j in {-7,...,2}
    {
    \draw[line width = 0.5 pt,red] ($(B)+0.3*(0,\i)$) -- ($(E)+0.15*(\j,0)$);
    }
}

\fitellipsiss{$(A)+(0.2,0.5)$}{$(A)-(0.2,0.5)$}{0.6};
\fitellipsiss{$(B)+(-0.2,0.5)$}{$(B)-(-0.2,0.5)$}{0.6};

\foreach \i in {-2,...,1}
{
    \foreach \j in {-4,...,4}
    {
    \draw[line width = 0.5 pt,red] ($(A)+(0.05,0.125)+0.2*(0,\i)$) -- ($(D)+(0.5,0)+0.2*(\j,0)$);
    }
}

\foreach \i in {-2,...,1}
{
    \foreach \j in {-4,...,4}
    {
    \draw[line width = 0.5 pt,red] ($(B)-(0.05,-0.125)+0.2*(0,\i)$) -- ($(D)+(-0.5,0)+0.2*(\j,0)$);
    }
}

\fitellipsiss{$(A)+(0.1,0.25)$}{$(A)-(0.1,0.25)$}{0.4};
\fitellipsiss{$(B)+(-0.1,0.25)$}{$(B)-(-0.1,0.25)$}{0.4};

\fitellipsiss{$(C)$}{$(E)$}{0.6};
\fitellipsiss{$(X)+(0.5,0)$}{$(X)-(0.5,0)$}{0.6};

\draw[] ($(X)+(0.6,-0.5)$) -- ($(X)+(0.6,0.5)$);
\draw[] ($(X)+(0,-0.6)$) -- ($(X)+(0,0.6)$);
\draw[] ($(X)+(-0.6,-0.5)$) -- ($(X)+(-0.6,0.5)$);

\draw[dashed] ($0.25*(C)+0.75*(D)+(0,-0.57)$) -- ($0.25*(C)+0.75*(D)+(0,0.57)$);
\draw[dashed] ($0.25*(E)+0.75*(D)+(0,-0.57)$) -- ($0.25*(E)+0.75*(D)+(0,0.57)$);

\draw[] ($0.4*(C)+0.6*(D)+(0,-0.57)$) -- ($0.4*(C)+0.6*(D)+(0,0.57)$);
\draw[] ($0.4*(E)+0.6*(D)+(0,-0.57)$) -- ($0.4*(E)+0.6*(D)+(0,0.57)$);

\draw [thick,decorate,decoration={brace,amplitude=7pt,raise=4pt,mirror},yshift=0pt]($0.4*(C)+0.6*(D)+(0,-0.57)$) -- ($0.4*(E)+0.6*(D)+(0,-0.57)$) node [black,midway,xshift=0cm,yshift=-0.6cm] {$D'$};

\draw [thick,decorate,decoration={brace,amplitude=5pt,raise=2pt,mirror},yshift=0pt]($1.3*(C)-0.3*(D)+(0,-0.17)$) -- ($0.25*(C)+0.75*(D)+(0,-0.57)$) node [black,midway,xshift=0cm,yshift=-0.5cm] {$C$};

\draw [thick,decorate,decoration={brace,amplitude=5pt,raise=2pt},yshift=0pt]($1.3*(E)-0.3*(D)+(0,-0.17)$) -- ($0.25*(E)+0.75*(D)+(0,-0.57)$) node [black,midway,xshift=0cm,yshift=-0.5cm] {$E$};

\pic[scale=0.15,rotate=15,blue] at ($(A)-(0.1,0.05)$) {conn};
\pic[scale=0.15,rotate=35,blue] at ($(B)-(0.1,0.15)$) {conn};
\pic[scale=0.1,red] at ($(X)-(0.9,0)$) {conn};
\pic[xscale=0.13, yscale=0.2,red] at ($(X)-(0.4,0.1)$) {conn};
\pic[scale=0.1,red] at ($(X)+(0.73,0)$) {conn};
\pic[xscale=0.13,yscale=0.2, red] at ($(X)+(0.2,-0.1)$) {conn};

\node[] at ($(X)+(-1.5,0)$) {$X$};
\node[] at ($(A)+(0.25,0.85)$) {$A'$};
\node[] at ($(B)+(-0.25,0.82)$) {$B'$};
\node[] at ($(A)+(-0.9,0.2)$) {$A$};
\node[] at ($(B)+(0.9,0.2)$) {$B$};
 \node[] at ($(D)+(0,0)$) {$D$};

\node[] at ($0.8*(C)+0.2*(D)+(0,0)$) {$C'$};
\node[] at ($0.8*(E)+0.2*(D)+(0,0)$) {$E'$};

\end{tikzpicture}
    \caption{Structure of the colouring w.r.t.\ $A'$ and $B'$.}
    \label{fig:overall3}
\end{minipage}\hfill
\begin{minipage}[t]{0.495\textwidth}
\centering
\captionsetup{width=\textwidth}
\begin{tikzpicture}[scale=1.1, rotate=90]
\defPt{0}{-0.9}{x1}
\defPt{0}{0.5}{x0}
\defPt{1.5}{-1.2}{x2}
\defPt{2}{1}{x3}
\defPt{2.5}{-0.6}{x4}

%\fitellipsiss{$(x3)$}{$(x1)$}{0.6};
\fitellipsiss{$(x2)$}{$(x4)$}{0.6};
\fitellipsiss{$(x3)+(0.5,-0.3)$}{$(x3)-(0.5,-0.3)$}{0.6};
\fitellipsiss{$(x1)+(0,0.1)$}{$(x1)-(0,0.5)$}{0.6};
\fitellipsiss{$(x0)+(0,0.1)$}{$(x0)+(0,0.7)$}{0.6};

\draw[line width = 2 pt, red]  (x3) -- (x0);
\draw[line width = 2 pt, red] (x0) -- (x4);
\draw[line width = 2 pt, red] (x4) -- (x1);
\draw[line width = 2 pt, red] (x1) -- (x2);

\draw[line width = 1 pt, blue] (x4) -- (x2);
\draw[line width = 1 pt, blue] (x4) -- (x3);
\draw[line width = 1 pt, blue] (x2) -- (x0);
\draw[line width = 1 pt, blue] (x1) -- (x3);
\draw[line width = 1 pt, blue] (x1) -- (x0);
\draw[line width = 1 pt, blue] (x2) -- (x3);

\node[] at ($(x1)+(-0.3,0)$) {$u$};
\node[] at ($(x0)+(-0.3,0)$) {$v$};
\node[] at ($(x2)+(-0.1,-0.25)$) {$b$};
\node[] at ($(x3)+(0.1,0.25)$) {$b'$};
\node[] at ($(x4)+(0.1,-0.25)$) {$r$};
\node[] at ($(x3)+(0.4,0.9)$) {$A'$};
\node[] at ($(x0)+(0,1.6)$) {$C_b'$};
\node[] at ($(x1)+(0,-1.4)$) {$E_b'$};
\node[] at ($0.5*(x2)+0.5*(x4)-(-0.4,0.9)$) {$B'$};

\foreach \i in {0,...,4}
{
\draw[] (x\i) \smvx;
}
\end{tikzpicture}
    \caption{$P_5$ leading to the contradiction in \Cref{cl:Cb-consistent-A-pr}.}
    \label{fig:Cb-consistent-A-pr}
\end{minipage}
\end{figure}

Now let us define $C_b'$ as the subset of vertices in $C'$ which have a blue neighbour in $E'$ and $C_r' = C'\setminus C_b'$, while similarly $E_b'$ is the set of vertices in $E'$ which have a blue neighbour in $C'$ and $E_r'= E'\setminus E_b'$. 
The following claim establishes restrictions on colours incident to these four new sets.
\begin{claim}\label{cl:Cb-consistent-A-pr} \textcolor{white}{something hidden}
\begin{enumerate}[label=\roman*)]
    \item $C_r'$ is red to $E'$,
    \item $E_r'$ is red to $C'$,
    \item\label{enum-iii)} $C_b'$ is blue to $B'$ and
    \item $E_b'$ is blue to $A'$.
\end{enumerate}
\end{claim}
\begin{cla_proof}
i) and ii) follow immediately from the definition of $C_b'$ and $E_b'$.

To see iii), assume for the sake of contradiction that some $v\in C_b'$ is not blue towards $B'$. Since $v \notin D'$, we also know $v$ is not red to $B'$, so is in particular not consistent to $B'$. Hence, by \Cref{lem:cute-nice-obs}, we can find vertices $r,b\in B'$ such that $rb$ is a blue edge (since $G[B']$ is connected in blue) with $vr$ being red and $vb$ being blue. We can find a $u\in E_b'$ such that $uv$ is blue by construction of $C_b'$
%(otherwise, $C_b$ and $E_b$ are empty and this claim is vacuously true)
and find $b'\in A$ that is blue to $u$, as $u \notin D'$. This implies $b'vrub$ is an induced red $P_5$, a contradiction. The proof of iv) follows symmetrically. See \Cref{fig:Cb-consistent-A-pr} for an illustration.
\end{cla_proof}

See \Cref{fig:overall4} for an illustration of the structure established by \Cref{cl:Cb-consistent-A-pr}. 

\begin{figure}
\RawFloats
\begin{minipage}[t]{0.495\textwidth}
\centering
\captionsetup{width=\textwidth}
    \begin{tikzpicture}[yscale=0.8]
\defPt{0}{0}{D}
\defPt{-2}{0}{C}
\defPt{2}{0}{E}
\defPt{-1.5}{2}{A}
\defPt{1.5}{2}{B}
\defPt{0}{4}{X}

%\draw[line width = 1.5 pt] (x3) -- (x4) -- (x1) -- (x2) -- (x0);

\foreach \i in {-2,-1,0,1,2}
{
    \foreach \j in {-1,0,1}
    {
    \draw[line width = 1 pt,blue] ($(X)+0.3*(\i,0)$) -- ($(A)+0.2*(\j,0)$);
    }
}

\foreach \i in {-2,-1,0,1,2}
{
    \foreach \j in {-1,0,1}
    {
    \draw[line width = 1 pt,blue] ($(X)+0.3*(\i,0)$) -- ($(B)+0.2*(\j,0)$);
    }
}

\foreach \i in {-2,-1,0,1,2}
{
    \foreach \j in {-2,-1,0,1,2}
    {
    \draw[line width = 1 pt,blue] ($(A)+0.3*(0,\i)$) -- ($(B)+0.3*(0,\j)$);
    }
}

\foreach \i in {-2,...,1}
{
    \foreach \j in {-2,...,7}
    {
    \draw[line width = 0.75 pt,red] ($(A)+0.3*(0,\i)$) -- ($(C)+0.15*(\j,0)$);
    }
}

\foreach \i in {-2,...,1}
{
    \foreach \j in {-7,...,2}
    {
    \draw[line width = 0.75 pt,red] ($(B)+0.3*(0,\i)$) -- ($(E)+0.15*(\j,0)$);
    }
}

\fitellipsiss{$(A)+(0.2,0.5)$}{$(A)-(0.2,0.5)$}{0.6};
\fitellipsiss{$(B)+(-0.2,0.5)$}{$(B)-(-0.2,0.5)$}{0.6};

\foreach \i in {-2,-1,0,1,1}
{
    \foreach \j in {-1,0,1}
    {
    \draw[line width = 1 pt,red] ($(A)+(0.05,0.125)+0.2*(0,\i)$) -- ($(D)+(0.5,0)+0.2*(\j,0)$);
    }
}

\foreach \i in {-2,-1,0,1,1}
{
    \foreach \j in {-1,0,1}
    {
    \draw[line width = 1 pt,red] ($(B)-(0.05,-0.125)+0.2*(0,\i)$) -- ($(D)+(-0.5,0)+0.2*(\j,0)$);
    }
}

\foreach \i in {-2,-1,0,1,2}
{
    \foreach \j in {-1,0,1}
    {
    \draw[line width = 0.5 pt,red] ($(C)-(0.1,0.1)+0.07*(\i,0)$) to[in=-60,out=-120] ($(E)+(0.1,-0.1)+0.07*(\j,0)$);
    }
}

\foreach \i in {-1,0,1}
{
    \foreach \j in {-1,0,1}
    {
    \draw[line width = 0.5 pt,red] ($(C)-(0.1,0.1)+0.07*(\i,0)$) to[in=-70,out=-110] ($(E)+(-0.87,-0.1)+0.07*(\j,0)$);
    }
}

\foreach \i in {-1,0,1}
{
    \foreach \j in {-1,0,1}
    {
    \draw[line width = 0.5 pt,red] ($(C)-(-0.92,0.1)+0.07*(\i,0)$) to[in=-70,out=-110] ($(E)+(0.1,-0.1)+0.07*(\j,0)$);
    }
}

\foreach \i in {-3,...,3}
{
    \foreach \j in {-1,...,1}
    {
    \draw[line width = 0.5 pt,blue] ($(A)+(0.05,0)+0.1*(0,\i)$) -- ($(E)-(0,0)+0.25*(\j,0)$);
    }
}

\foreach \i in {-3,...,3}
{
    \foreach \j in {-1,0,1}
    {
    \draw[line width = 0.5 pt,blue] ($(B)-(0.05,0)+0.1*(0,\i)$) -- ($(C)+(0,0)+0.25*(\j,0)$);
    }
}

\fitellipsiss{$(A)+(0.1,0.25)$}{$(A)-(0.1,0.25)$}{0.4};
\fitellipsiss{$(B)+(-0.1,0.25)$}{$(B)-(-0.1,0.25)$}{0.4};

\fitellipsiss{$(C)$}{$(E)$}{0.6};
\fitellipsiss{$(X)+(0.5,0)$}{$(X)-(0.5,0)$}{0.6};

\draw[] ($(X)+(0.6,-0.5)$) -- ($(X)+(0.6,0.5)$);
\draw[] ($(X)+(0,-0.6)$) -- ($(X)+(0,0.6)$);
\draw[] ($(X)+(-0.6,-0.5)$) -- ($(X)+(-0.6,0.5)$);

\draw[] ($0.3*(C)+0.7*(D)+(0,-0.57)$) -- ($0.3*(C)+0.7*(D)+(0,0.57)$);
\draw[] ($0.3*(E)+0.7*(D)+(0,-0.57)$) -- ($0.3*(E)+0.7*(D)+(0,0.57)$);

\draw[] ($0.8*(C)+0.2*(D)+(0,-0.47)$) -- ($0.8*(C)+0.2*(D)+(0,0.47)$);
\draw[] ($0.8*(E)+0.2*(D)+(0,-0.47)$) -- ($0.8*(E)+0.2*(D)+(0,0.47)$);

\pic[scale=0.15,rotate=15,blue] at ($(A)-(0.1,0.05)$) {conn};
\pic[scale=0.15,rotate=35,blue] at ($(B)-(0.1,0.15)$) {conn};
\pic[scale=0.1,red] at ($(X)-(0.9,0)$) {conn};
\pic[xscale=0.13, yscale=0.2,red] at ($(X)-(0.4,0.1)$) {conn};
\pic[scale=0.1,red] at ($(X)+(0.73,0)$) {conn};
\pic[xscale=0.13,yscale=0.2, red] at ($(X)+(0.2,-0.1)$) {conn};

\node[] at ($(X)+(-1.5,0)$) {$X$};
\node[] at ($(A)+(0.25,0.85)$) {$A'$};
\node[] at ($(B)+(-0.25,0.82)$) {$B'$};
\node[] at ($(A)+(-0.9,0.2)$) {$A$};
\node[] at ($(B)+(0.9,0.2)$) {$B$};
\node[] at ($(C)-(0,0)$) {$C_r'$};
\node[] at ($(C)+(0.92,0)$) {$C_b'$};
\node[] at ($(D)$) {$D'$};
\node[] at ($(E)-(0.87,0)$) {$E_b'$};
\node[] at ($(E)+(0,0)$) {$E_r'$};

\end{tikzpicture}%
    \caption{Structure established by i)--iv).}
    \label{fig:overall4}
\end{minipage}\hfill
\begin{minipage}[t]{0.495\textwidth}
    \centering
    \captionsetup{width=\textwidth}
    \begin{tikzpicture}[scale=1.1, rotate=90]
\defPt{-0.3}{-.9}{x4}
\defPt{0}{1}{x0}
\defPt{2}{-1.2}{x2}
\defPt{2}{1}{x3}
\defPt{0}{-1.5}{x1}

%\fitellipsiss{$(x3)$}{$(x1)$}{0.6};
%\fitellipsiss{$(x2)$}{$(x4)$}{0.6};
\fitellipsiss{$(x2)-(0.5,0.3)$}{$(x2)+(0.5,0.3)$}{0.6};
\fitellipsiss{$(x3)+(0.5,-0.3)$}{$(x3)-(0.5,-0.3)$}{0.6};
\fitellipsiss{$(x1)+(0,0.6)$}{$(x1)-(0,0)$}{0.6};
\fitellipsiss{$(x0)+(0,-0.3)$}{$(x0)+(0,0.3)$}{0.6};

\draw[line width = 2 pt, red]  (x3) -- (x0);
\draw[line width = 2 pt, red] (x0) -- (x4);
\draw[line width = 2 pt, red]  (x2) -- (x4);
\draw[line width = 2 pt, red] (x2) -- (x1);

\draw[line width = 1 pt, blue]  (x3) -- (x2);
\draw[line width = 1 pt, blue]  (x0) -- (x2);
\draw[line width = 1 pt, blue]  (x4) -- (x1);
\draw[line width = 1 pt, blue] (x4) -- (x3);
\draw[line width = 1 pt, blue] (x1) -- (x0);
\draw[line width = 1 pt, blue] (x1) -- (x3);

\node[] at ($(x1)+(-0.1,-0.25)$) {$b$};
\node[] at ($(x0)+(-0.3,0)$) {$v$};
\node[] at ($(x2)+(-0.1,-0.25)$) {$w$};
\node[] at ($(x3)+(0.1,0.25)$) {$u$};
\node[] at ($(x4)+(-0.1,-0.25)$) {$r$};
\node[] at ($(x3)+(0.4,0.9)$) {$A$};
\node[] at ($(x0)+(0,1.2)$) {$C_b$};
\node[] at ($(x4)+(0.3,-1.5)$) {$E_b$};
\node[] at ($(x2)-(-0.4,0.9)$) {$B'$};

\foreach \i in {0,...,4}
{
\draw[] (x\i) \smvx;
}
\end{tikzpicture}
    \caption{$P_5$ leading to the contradiction in \Cref{cl:blue-compo-cb-eb}.}
    \label{fig:blue-compo-cb-eb}
\end{minipage}
\end{figure}

\begin{claim}\label{cl:blue-compo-cb-eb}
Any blue connected component of $G[C_b']$ forms a pure pair with any blue connected component of $G[E_b']$.
\end{claim}
\begin{cla_proof}
Pick a blue component $C_b$ of $C_b'$ and a blue component $E_b$ of $E_b'$. Assume for a contradiction that $v\in C_b$ is not consistent to $E_b$. Then, using \Cref{lem:cute-nice-obs} we can find $b,r\in E_b$ such that $rb$ is a blue edge with $vr$ red and $vb$ blue. From iv) above, if we pick an arbitrary $u\in A'$, it will be blue to $E_b' \supseteq \{b,r\}$ and red to $C_b' \ni v$. Similarly, we can pick $w\in B'$ that is red to $\{b,r\}$ and blue to $v$. Since $uw$ is blue, this gives us that $uvrwb$ induces a $P_5$ in red, a contradiction. We can do a symmetric proof to show $E_b$ is consistent to $C_b$. Having $C_b$ consistent to $E_b$ and vice versa guarantees they form a pure pair, as desired. See \Cref{fig:blue-compo-cb-eb} for an illustration.
\end{cla_proof}

\begin{claim}\label{cl:x-b-e-pr-bound}
$|X| + |B| + |E'| \ge 6n^{1-\eps_n^2}$ and 
$|X|+|A|+|C'| \ge 6n^{1-\eps_n^2}$.
\end{claim}

\begin{cla_proof}
Suppose towards a contradiction that 
$|X|+|B|+|E_b'|+|E_r'| < 6n^{1-\eps_n^2}$. This implies a lower bound on the size of the remaining sets of vertices of our subgraph $G'$:
$$|A'|+|A \setminus A'|+|D'|+|C_b'|+|C_r'|\ge |G'| -6n^{1-\eps_n^2} \ge n -7n^{1-\eps_n^2}.$$ 

This combined with the observation that $(A',(A \setminus A')\cup D' \cup C_b'\cup C_r')$ make a red pure pair and \eqref{eq:a-prime-bnd}, which gives $|A'|\ge n^{1-\eps_n}\ge n^{1-\eps_n/c}$, implies through our second ``win'' condition, namely \Cref{cl:win-purepair-2} that 
$$|A \setminus A'|+|D'|+|C_b'|+|C_r'| \le n^{1-\eps_n/c}.$$

Combining the two above bounds implies 
$$|A'| \ge n-7n^{1-\eps_n^2}-n^{1-\eps_n/c}.$$ 
We deduce from \eqref{eq:a-prime-bnd} that $|B'|\ge n^{1-\eps_n} \ge n^{1-\eps_n/c}$ and combine it with the previous inequality to obtain $|A'|+|B'| \ge n-7n^{1-\eps_n^2}$. As we already observed (based on \eqref{eq:a-prime-bnd}) that $|A'|\ge n^{1-\eps_n}$, this results in a contradiction to our second win condition, \Cref{cl:win-purepair-2}, applied to the pure pair $(A',B')$.

Since we did not yet use the assumption that $|A|\le |B|$, the argument can be repeated symmetrically to obtain the second claim.
\end{cla_proof}

We will now break symmetry and start using the assumption that $|A| \le |B|$. Our next goal is to show that this implies $B$ and $C_r'$ contain most of our graph $G'$, namely that $C_r' \cup B$ is huge and that both $C_r'$ and $B$ and are large.

\begin{claim}\label{cl:bounds}
We have $|A|,|X|,|E'|,|B\setminus B'|,|C_b'|,|D'| < n^{1-8\eps_n^2}$ and 
$|B|,|C_r'|\ge 2n^{1-\eps_n^2}$. Additionally, $|B'|\ge n^{1-\eps_n^2}$.
\end{claim}
\begin{cla_proof}
The first desired inequality is a consequence of the assumption $|A|\le |B|$ and \Cref{cl:win-purepair-1}, since $(A,B)$ is a pure pair, which implies 
$$|A|< n^{1-\eps_n^2/c} = n^{1-16\eps_n^2}.$$
The second one comes by combining this with \Cref{cl:bound-on-X}: 
$$|X|<|A| \cdot 2n^{c\eps_n^2}\le n^{1-8\eps_n^2}.$$%$2\le n^{\eps_n^2(1/(2c)-1-c^2)}     
This, in turn, combined with the first part of \Cref{cl:x-b-e-pr-bound}, shows that
\begin{equation}\label{eq:b-e-pr-bound}
    |B| + |E'| = |B| + |E_b'| + |E_r'| 
    \ge 6n^{1-\eps_n^2}-n^{1-8\eps_n^2},
\end{equation}
while the second part of \Cref{cl:x-b-e-pr-bound} and the above upper bounds on $|A|$ and $|X|$ give
\begin{equation}\label{eq:c-pr-bnd}
    |C'| = |C_b'| + |C_r'| \ge 5n^{1-\eps_n^2}.
\end{equation}
Furthermore, since $(C',E_r')$ is a red pair, we can deduce from our first win condition, \Cref{cl:win-purepair-1}, and \eqref{eq:c-pr-bnd} that 
\begin{equation}\label{eq:e-r-pr-bound}
    |E_r'| < n^{1-16\eps_n^2}.
\end{equation}
This, in turn, refines \eqref{eq:b-e-pr-bound} to give us
\begin{equation}\label{eq:b-eb-pr-bound}
    |B| + |E_b'| \ge 6n^{1-\eps_n^2}-2n^{1-8\eps_n^2}
\end{equation}
Suppose first towards a contradiction that $|E_b'|\ge 3n^{1-16\eps_n^2}.$ Since by \Cref{cl:Cb-consistent-A-pr} part (i) we have that $(E_b',C_r')$ is a red pair, by our first win condition (\Cref{cl:win-purepair-1}) we have $|C_r'|< n^{1-16\eps_n^2}$ which when combined with \eqref{eq:c-pr-bnd} implies $|C_b'|\ge 3n^{1-16\eps_n^2}$. Now using \Cref{lem:one-third-pure-pair} on both the blue graphs induced by $E_b'$ and $C_b'$, we conclude that either we find a pure pair with both parts of size at least $n^{1-16\eps_n^2},$ so sufficient for \Cref{cl:win-purepair-1} to give a contradiction, or the size of the largest blue component of both of these graphs is at least $\frac13$ of the size of the set itself, so $n^{1-16\eps_n^2}$. However, by \Cref{cl:blue-compo-cb-eb}, these two components make a pure pair which once again has sufficient size for \Cref{cl:win-purepair-1} to give a contradiction. Therefore, we have $|E_b'|< 3 n^{1-16\eps_n^2}.$ 

Using this bound combined with the above upper bound, \eqref{eq:e-r-pr-bound}, on $|E_r'|$ shows $|E'|=|E_r' \cup E_b'|\le 4n^{1-16\eps_n^2}< n^{1-8\eps_n^2}$. On the other hand, combined with \eqref{eq:b-eb-pr-bound} it implies $|B|\ge 5n^{1-\eps_n^2}$, which satisfies one of our desired bounds.

Here, once again, we can conclude via \Cref{lem:one-third-pure-pair} that $|B'|\ge \frac53n^{1-\eps_n^2},$ establishing the desired lower bound on $|B'|$.

Since $(B',B \setminus B'), (B',D')$ and $(B',C_b')$ are all pure pairs by \Cref{cl:win-purepair-1} we get the desired upper bounds $|B \setminus B'|$, $|D'|,|C_b'|\le n^{1-16\eps_n^2}.$  Combining the last of these with \eqref{eq:c-pr-bnd} gives the desired lower bound on $C_r'$. 
\end{cla_proof}

A key consequence of the previous claim is that since the eight sets considered in the previous claim partition the vertex set of $G'$, we must have that $C_r'$ and $B'$ contain a vast majority of the vertices in our graph:
\begin{equation}\label{eq:almost-evrything}
    |C_r'|+|B'| \ge |G'|- 6n^{1-8\eps_n^2}\ge n - 7n^{1-\eps_n^2}.
\end{equation}
% Furthermore, $E_r'$ makes a red pair with $C_r'\cup B'$ by construction, which gives us via \Cref{cl:win-purepair-2}
% \begin{equation}
%     |E_r'| < \frac{n}{2^{19c(\log n)^{2/3}}} \le |A|.
% \end{equation}

Before turning to the next claim, let us introduce $B''$ as the largest connected component in the red graph induced on $B'$. By \Cref{lem:one-third-pure-pair}, we know $|B''|\ge \frac13|B'|$, as otherwise we are done by the first win condition \Cref{cl:win-purepair-1}. This combined with \Cref{cl:bounds} gives
\begin{equation*}
|B''|\ge \frac 13 |B'|  \ge \frac 13 n^{1-\eps_n^2} > n^{1-8\eps_n^2}> |A|.    
\end{equation*}

Our next observation is that $C_b'$ and $E_b'$ must be empty. 
\begin{claim}\label{cl:cb-empty}
$C_b' = \emptyset$ and $E_b' = \emptyset$.
\end{claim}
\begin{cla_proof}
We first show that $C_b'=\emptyset$. Suppose towards a contradiction that this is not the case. Observe that $B'$ is blue to $A\cup C_b'$ using \Cref{cl:Cb-consistent-A-pr} part \ref{enum-iii)} and since $C_b' \subseteq C'$ we know $C_b'$ is red to $A$ so, in particular, $G[A \cup C_b']$ is connected in red. This and $B''$ being connected in red implies $(A\cup C_b',\ B'')$ is a good pair and since $|B''|>|A|$ this pair contradicts the maximality of $(A, B)$ since $|A \cup C_b'|,|B'|>|A|$ (using that $C_b'\neq \emptyset$). This means there are no vertices in $C'$ which have a blue neighbour in $E'$, so also there are no vertices in $E'$ with a blue neighbour in $C'$, implying also that $E_b' = \emptyset.$
\end{cla_proof}

This shows $C'=C_r'$ and $E'=E_r'$. The most problematic of our remaining sets is $D'$, since we do not know anything about the edges between $D'$ and our large, and potentially huge set $C'$. The following lemma will tell us that while this may be true, we can split it into two parts, one of which is red to $C'$ and one which is red to the whole of $A$, as opposed just to $A'$.
To this end, let $D''$ be the set of vertices in $D'$, which send a blue edge towards $A$. See \Cref{fig:overall5} for an illustration.

\begin{figure}
\RawFloats
\begin{minipage}[t]{0.495\textwidth}
\centering
\captionsetup{width=\textwidth}
    \begin{tikzpicture}[yscale=0.8]
\defPt{0}{0}{D}
\defPt{-2}{0}{C}
\defPt{2}{0}{E}
\defPt{-1.5}{2}{A}
\defPt{1.5}{2}{B}
\defPt{0}{4}{X}

\foreach \i in {-2,-1,0,1,2}
{
    \foreach \j in {-1,0,1}
    {
    \draw[line width = 1 pt,blue] ($(X)+0.3*(\i,0)$) -- ($(A)+0.2*(\j,0)$);
    }
}

\foreach \i in {-2,-1,0,1,2}
{
    \foreach \j in {-1,0,1}
    {
    \draw[line width = 1 pt,blue] ($(X)+0.3*(\i,0)$) -- ($(B)+0.2*(\j,0)$);
    }
}

\foreach \i in {-2,-1,0,1,2}
{
    \foreach \j in {-2,-1,0,1,2}
    {
    \draw[line width = 1 pt,blue] ($(A)+0.3*(0,\i)$) -- ($(B)+0.3*(0,\j)$);
    }
}

\foreach \i in {-7,...,6}
{
    \foreach \j in {2,...,4}
    {
    \draw[line width = 0.75 pt,red] ($(A)+(0.05,0.125)+0.75*(0.06*\i,0.2*\i)$) -- ($(D)+(-0.7,0)+0.2*(\j,0)$);
    }
}

\foreach \i in {-3,...,3}
{
    \foreach \j in {-8,...,-1}
    {
    \draw[line width = 0.25 pt,red] ($(A)+(0.05,0.125)+0.75*(0.06*\i,0)$) -- ($(D)+(-1.25,0)+0.15*(\j,0)$);
    }
}

\foreach \i in {-4,...,4}
{
    \foreach \j in {-3,...,2}
    {
    \draw[line width = 0.5 pt,red] ($(B)+0.3*(-0.4*\i,0.8*\i)$) -- ($(E)+0.15*(\j,0)$);
    }
}

\fitellipsiss{$(A)+(0.2,0.5)$}{$(A)-(0.2,0.5)$}{0.6};
\fitellipsiss{$(B)+(-0.2,0.5)$}{$(B)-(-0.2,0.5)$}{0.6};

\foreach \i in {-4,...,4}
{
    \foreach \j in {2,...,4}
    {
    \draw[line width = 0.3 pt,red] ($(A)+(0.05,0.125)+0.1*(0,\i)$) -- ($(D)+(1,0)+0.2*(\j,0)$);
    }
}

\foreach \i in {-2,...,2}
{
    \foreach \j in {1,...,4}
    {
    \draw[line width = 0.5 pt,red] ($(B)-(0.05,-0.125)+0.2*(0,\i)$) -- ($(D)+(0,0)+0.2*(\j,0)$);
    }
}

\foreach \i in {-4,...,4}
{
    \foreach \j in {2,...,4}
    {
    \draw[line width = 0.3 pt,red] ($(B)-(0.05,-0.125)+0.1*(0,\i)$) -- ($(D)+(-2,0)+0.2*(\j,0)$);
    }
}

\foreach \i in {-2,...,2}
{
    \foreach \j in {-2,...,2}
    {
    \draw[line width = 0.2 pt,red] ($(C)-(-0.1,0.1)+0.07*(\i,0)$) to[in=-70,out=-110] ($(E)+(-1.17,-0.1)+0.07*(\j,0)$);
    }
}

\foreach \i in {-2,...,2}
{
    \foreach \j in {-2,...,2}
    {
    \draw[line width = 0.2 pt,red] ($(C)-(-0.1,0.1)+0.07*(\i,0)$) to[in=-70,out=-110] ($(E)+(-0.1,-0.1)+0.07*(\j,0)$);
    }
}

\fitellipsiss{$(A)+(0.1,0.25)$}{$(A)-(0.1,0.25)$}{0.4};
\fitellipsiss{$(B)+(-0.1,0.25)$}{$(B)-(-0.1,0.25)$}{0.4};

\fitellipsiss{$(C)$}{$(E)$}{0.6};
\fitellipsiss{$(X)+(0.5,0)$}{$(X)-(0.5,0)$}{0.6};

\draw[] ($(X)+(0.6,-0.5)$) -- ($(X)+(0.6,0.5)$);
\draw[] ($(X)+(0,-0.6)$) -- ($(X)+(0,0.6)$);
\draw[] ($(X)+(-0.6,-0.5)$) -- ($(X)+(-0.6,0.5)$);

\draw[] ($0.6*(C)+0.4*(D)+(0,-0.52)$) -- ($0.6*(C)+0.4*(D)+(0,0.52)$);
\draw[] ($0.6*(E)+0.4*(D)+(0,-0.52)$) -- ($0.6*(E)+0.4*(D)+(0,0.52)$);
\draw[] ($(D)+(0,-0.6)$) -- ($(D)+(0,0.6)$);

\pic[scale=0.15,rotate=15,blue] at ($(A)-(0.1,0.05)$) {conn};
\pic[scale=0.15,rotate=35,blue] at ($(B)-(0.1,0.15)$) {conn};
\pic[scale=0.1,red] at ($(X)-(0.9,0)$) {conn};
\pic[xscale=0.13, yscale=0.2,red] at ($(X)-(0.4,0.1)$) {conn};
\pic[scale=0.1,red] at ($(X)+(0.73,0)$) {conn};
\pic[xscale=0.13,yscale=0.2, red] at ($(X)+(0.2,-0.1)$) {conn};

\node[] at ($(X)+(-1.5,0)$) {$X$};
\node[] at ($(A)+(0.25,0.85)$) {$A'$};
\node[] at ($(B)+(-0.25,0.82)$) {$B'$};
\node[] at ($(A)+(-0.9,0.2)$) {$A$};
\node[] at ($(B)+(0.9,0.2)$) {$B$};
\node[] at ($(D)+(-0.59,-0.05)$) {\small $D'\setminus D''$};
\node[] at ($(D)+(0.59,0)$) {$D''$};

\node[] at ($0.9*(C)+0.1*(D)+(0,0)$) {$C'$};
\node[] at ($0.9*(E)+0.1*(D)+(0,0)$) {$E'$};

\end{tikzpicture}
    \caption{Final structure of the colouring}
    \label{fig:overall5}
\end{minipage}\hfill
\begin{minipage}[t]{0.495\textwidth}
\centering
\captionsetup{width=\textwidth}
\begin{tikzpicture}[scale=1.1, rotate=90]
\defPt{0}{-0.1}{x1}
\defPt{0}{1.3}{x0}
\defPt{1.5}{-1.2}{x2}
\defPt{2}{1}{x3}
\defPt{2.5}{-0.6}{x4}

\fitellipsiss{$(x2)$}{$(x4)$}{0.6};
\fitellipsiss{$(x3)+(0.5,-0.3)$}{$(x3)-(0.5,-0.3)$}{0.6};
\fitellipsiss{$(x1)+(0,0.2)$}{$(x1)-(0,0.4)$}{0.6};
\fitellipsiss{$(x0)+(0,0.2)$}{$(x0)+(0,0.8)$}{0.6};

\draw[line width = 2 pt, red]  (x3) -- (x0);
\draw[line width = 2 pt, red] (x0) -- (x4);
\draw[line width = 2 pt, red] (x4) -- (x1);
\draw[line width = 2 pt, red] (x1) -- (x2);

\draw[line width = 1 pt, blue] (x4) -- (x2);
\draw[line width = 1 pt, blue] (x4) -- (x3);
\draw[line width = 1 pt, blue] (x2) -- (x0);
\draw[line width = 1 pt, blue] (x1) -- (x3);
\draw[line width = 1 pt, blue] (x1) -- (x0);
\draw[line width = 1 pt, blue] (x2) -- (x3);

\node[] at ($(x1)+(-0.3,0)$) {$y$};
\node[] at ($(x0)+(-0.3,0)$) {$x$};
\node[] at ($(x2)+(-0.1,-0.25)$) {$b_2$};
\node[] at ($(x3)+(0.1,0.25)$) {$a$};
\node[] at ($(x4)+(0.1,-0.25)$) {$b_1$};
\node[] at ($(x3)+(0.4,0.9)$) {$A$};
\node[] at ($(x0)+(0,1.7)$) {$C'$};
\node[] at ($(x1)+(0,-1.3)$) {$D''$};
\node[] at ($0.5*(x2)+0.5*(x4)-(-0.4,0.9)$) {$B'$};

\foreach \i in {0,...,4}
{
\draw[] (x\i) \smvx;
}
\end{tikzpicture}
    \caption{$P_5$ leading to the contradiction in \Cref{claim-figure-1}}
    \label{fig:claim-figure-1}
\end{minipage}
\end{figure}

\begin{claim}\label{claim-figure-1}
$(D'', C'\cup B')$ is a red pair.
\end{claim}
\begin{cla_proof}
We have that $D''\subseteq D'$, so $D''$ is red to $B'$. It only remains to show that $D''$ is red to $C'$. Suppose towards a contradiction that there is a blue edge $xy$ with $x\in C'$ and $y\in D''$. By construction of $C'$, $x\in C'$ is not red to $B'$. Furthermore, $x$ can't be blue to $C'$ since otherwise, similarly as in the previous claim, $(A \cup \{x\},B'')$ would be a good pair with both parts of size larger than $A$, contradicting maximality of $(A,B)$. This means $x$ is inconsistent to $B'$ (which is connected in blue), consequently via \Cref{lem:cute-nice-obs} we can find vertices $b_1,b_2\in B'$ such that $b_1x$ is red, $b_2x$ is blue, and $b_1b_2$ is blue. Let $a\in A$ be a blue neighbour of $y$, which exists by the definition of $D''$. Then, $axb_1yb_2$ induces a red $P_5$. See \Cref{fig:claim-figure-1} for an illustration.
\end{cla_proof}

%\begin{center}
%\includegraphics[scale=0.25]{images/d-prime-red.png}    
%\end{center}

By its definition, we have that $E'$ is red to $B'$ and since $E'=E_r'$ by \Cref{cl:cb-empty}, we also know $E'$ is red to $C'$. Hence, by the previous claim, $(E'\cup D'', B'\cup C')$ is a red pair. Combining this with \eqref{eq:almost-evrything} and applying our second win condition \Cref{cl:win-purepair-2}, we conclude that
\begin{equation}\label{eq:Dprpr-Er-pr-tiny}
    |E'|+|D''| < n^{1-\eps_n/c}<|A|.
\end{equation}
Noticing that every vertex outside of $E'$ and $D''$ is consistent to $A$, this establishes one of our key goals. Namely, we found a set of at least medium size with the property that all but a small number of vertices are consistent to it. The final ingredient in the proof of our main claim is to show that our maximality assumption on $(A,B)$ guarantees small blue degrees from any vertex in $C \cup D \cup E$ to $B \cup X$. %which we could have established earlier but would have been perhaps less motivated there as a sensible reader would have been thinking of B of being of size similar to A.  is that the maximum blue degree from $C' \cup (D' \setminus D'')$ to $B \cup X$ should be small. In other words, the pair $(C' \cup (D' \setminus D''),B \cup X)$ is predominantly red. This is a key step in which we use maximality of the good pair $(A,B)$.

\begin{claim}\label{cl:small-deg-R0-to-B0}
The blue degree of any vertex $u \in C \cup D \cup E$ towards $B \cup X$ is less than 
$2n^{c\eps_n^2}|A|\le n^{\eps_n^2}|A|.$
\end{claim}
\begin{cla_proof}
Observe that, by \Cref{cl:X-claim}, $u$ sends a red edge towards $A$ since $u \notin X$, so $G[A \cup \{u\}]$ is connected in red. On the other hand, the set $B \cup X$ is blue to $A$. So if we let $N$ denote the blue neighbourhood of $u$ within $B \cup X$, there can be no red connected component $Y$ within $N$ of size larger than $A$ since otherwise $A \cup \{u\}$ would make a good pair with $Y$ which has both sides larger than $|A|$, contradicting maximality of the good pair $(A,B)$. See \Cref{fig:blue-deg} for an illustration.

This implies that every red connected component in $G[N]$ must have size at most $|A|$. This implies via \Cref{r-partite-bound} that 
\begin{align*}
    n^{c\eps_n}= f(n)> \floor{\frac{|N|}{|A|}} \cdot f(|A|) &\ge \frac{|N|}{2|A|}\cdot |A|^{c\eps_n}
                       \ge \frac{|N|}{2|A|} \cdot \left(n^{1-\eps_n}\right)^{c\eps_n}
                       = \frac{|N|}{2|A|} \cdot n^{c\eps_n-c\eps_n^2},
\end{align*}
where in the first inequality, we used the assumption that $|N| \ge |A|$, which we may assume as otherwise, the claim holds immediately. Rearranging the inequality and cancelling out the $n^{c\eps_n}$ gives the first desired inequality. The second follows from $c\eps_n^2 \le \eps_n^2/2$ and $2\le 1/\eps_n= (\log n)^{1/3}$.
\end{cla_proof}

\begin{figure}
\RawFloats
\begin{minipage}[t]{0.495\textwidth}
\captionsetup{width=\textwidth}
\centering
\begin{tikzpicture}[yscale=0.8]
\defPt{0}{0}{D}
\defPt{-2}{0}{C}
\defPt{2}{0}{E}
\defPt{-1.5}{2}{A}
\defPt{1.5}{2}{B}

\foreach \i in {-2,...,2}
{
    \foreach \j in {-3,...,3}
    {
    \draw[line width = 1 pt,blue] ($(A)+0.3*(0,\i)$) -- ($(B)+0.3*(0,\j)$);
    }
}

\fitellipsiss{$(C)$}{$(E)$}{0.6};
\fitellipsiss{$(A)+(0,0.4)$}{$(A)-(0,0.4)$}{0.4};
\fitellipsiss{$(B)+(-0,0.4)$}{$(B)-(-0,0.4)$}{0.9};

\foreach \i in {-2,...,2}
{
    \draw[line width = 1 pt,blue] ($(B)+0.4*(-0.1*\i,\i)$) -- ($(D)-(1,0)$);
}
\draw[line width = 1 pt,blue] ($(B)+(-0.4,0.8)$) -- ($(D)-(1,0)$);

\draw[line width = 1 pt,red] ($(A)+(-0.09,-0.2)$) -- ($(D)-(1,0)$);

\fitellipsiss{$(B)+(-0,0.5)$}{$(B)-(-0,0.5)$}{0.6};

\fitellipsiss{$(B)+(-0,0.21)$}{$(B)-(-0,0.6)$}{0.4};

\draw[] ($(D)-(1,0)$) \vx;

\pic[scale=0.18,rotate=15,red] at ($(A)-(0.15,0.15)$) {conn};

\pic[scale=0.18,rotate=15,red] at ($(B)-(0.15,0.35)$) {conn};

\node[] at ($(A)+(-0.9,0.1)$) {$A$};
\node[] at ($(B)+(0,0.83)$) {\small $N$};
\node[] at ($(B)+(1.5,0.1)$) {$B \cup X$};
\node[] at ($(C)+(-1.6,0)$) {$C \cup D \cup E$};
\node[] at ($(D)+(-1,-0.3)$) {$u$};

\end{tikzpicture}
    \caption{Illustration of the proof of \Cref{cl:small-deg-R0-to-B0}.}
    \label{fig:blue-deg}
\end{minipage}\hfill
\begin{minipage}[t]{0.495\textwidth}
\centering
\begin{tikzpicture}[yscale=0.8]
\defPt{0}{0}{D}
\defPt{-2}{0}{C}
\defPt{2}{0}{E}
\defPt{-1.5}{2}{A}
\defPt{1.5}{2}{B}

\fitellipsiss{$(C)$}{$(E)$}{1.5};
\draw[] ($(C)-(1.42,-0.3)$) -- ($(E)+(-1.5,0.3)$);
\draw[] ($(E)-(1.5,1.48)$) -- ($(E)+(-1.5,1.48)$);
\draw[] ($(C)+(-0.4,-1.1)$) -- ($(C)+(-0.4,1.1)$);
\draw[] ($(C)+(0.5,-1.36)$) -- ($(C)+(0.5,1.36)$);
\draw[] ($(C)+(1.6,-1.48)$) -- ($(C)+(1.6,1.48)$);

\foreach \i in {-6,-4,-2,0,2,4,6}
{
    \foreach \j in {-2,...,2}
    {
    \draw[line width = 0.1 pt,red] ($(E)-(1.5,0.5)+0.05*(0,\i)$) -- ($(C)+(3.2,0.5)+0.1*(\j,0)$);
    }
}

\foreach \i in {-6,...,6}
{
    \foreach \j in {-3,...,3}
    {
    \draw[line width = 0.5 pt,blue] ($(E)-(1.5,0.48)+0.05*(0,\i)$) -- ($(C)+(3.8,-0.5)+0.1*(\j,0)$);
    }
}

\foreach \i in {-4,...,4}
{
    \foreach \j in {-4,...,4}
    {
    \draw[line width = 0.1 pt,red] ($(C)+(3.8,0.5)+0.2*(\i,0)$) -- ($(C)+(3.8,-0.5)+0.2*(\j,0)$);
    }
}

\foreach \i in {-1,0,1}
{
    \foreach \j in {-1,0,1}
    {
    \draw[line width = 0.1 pt,blue] ($(C)+(3.8,0.5)+0.8*(\i,0)$) -- ($(C)+(3.8,-0.5)+0.8*(\j,0)$);
    }
}

\fitellipsiss{$(C)+(3.2,0.5)$}{$(E)+(0.5,0.5)$}{0.4};
\fitellipsiss{$(C)+(3.2,-0.5)$}{$(E)+(0.5,-0.5)$}{0.4};

\node[] at ($(C)+(-1.5,1)$) {$G_0=G$};
\node[] at ($(C)+(-0.7,-0.4)$) {$A_1$};
\node[] at ($(C)+(-0.7,0.6)$) {$E_1$};
\node[] at ($(C)+(0.05,-0.4)$) {$A_2$};
\node[] at ($(C)+(0.05,0.6)$) {$E_2$};
\node[] at ($(C)+(1.08,-0.4)$) {$\cdots$};
\node[] at ($(C)+(1.08,0.6)$) {$\cdots$};
\node[] at ($(C)+(2.05,-0.4)$) {$A_i$};
\node[] at ($(C)+(2.05,0.6)$) {$E_i$};
\node[] at ($(C)+(3.8,0.5)$) {$R_i$};
\node[] at ($(C)+(3.8,-0.5)$) {$B_i$};

\draw [thick,decorate,decoration={brace,amplitude=9pt,raise=4pt,mirror},yshift=3pt] ($(C)+(1.55,-1.5)$) -- ($(E)+(1.5,-0.57)$) node [black,midway,xshift=0.4cm,yshift=-0.6cm] {$G_i$};

\end{tikzpicture}

\captionsetup{width=\textwidth}
 \caption{Illustration of the iteration.}
    \label{fig:iteration}
\end{minipage}

\end{figure}

We are now ready to define the partition required by our main claim. Indeed, we choose $A(G') := A,\ Red(G') := C'\cup (D'\setminus D''),\ Blu(G') := B\cup X,$ and $Err(G') := D''\cup E'$. Note that this is indeed a partition of $V(G')$ by \Cref{cl:cb-empty} and note that $Red(G')$ is red to $A(G')$ by definition of the sets involved, while $Blu(G')=B \cup X$ is blue to $A$. So they satisfy the colouring requirements of our main claim. Note also that $|Err|=|D''\cup E'| < |A|$ by \eqref{eq:Dprpr-Er-pr-tiny} establishing property \ref{itm:d} of the main claim. Property \ref{itm:a} follows from \Cref{eq:A-size-bound} while the required size conditions follow from the lower bounds part of \Cref{cl:bounds}. The final remaining requirement concerning the blue degree from $Red(G')$ to $Blu(G')$ is satisfied thanks to \Cref{cl:small-deg-R0-to-B0} and completes the proof of the main claim.

We are now ready to iteratively use the main claim to complete the proof. We start by setting $G_0=G$ and then for all $i$, we define $A_i=A(G_i), B_i=Blu(G_i), R_i=Red(G_i)$ and $E_i=Err(G_i)$ and define $G_{i+1}=G_i \setminus (A_i \cup E_i)$ so long as $|G_i|\ge n-n^{1-\eps_n^2}$, which allows us to apply the main claim and find the partition. Let $t$ be the final value of $i$ in this process so that $|G_{t-1}|\ge n-n^{1-\eps_n^2}$ and $|G_t|< n-n^{1-\eps_n^2}$. This, in particular, guarantees 
\begin{equation}\label{eq:leftover-small}
|A_0 \cup \ldots \cup A_{t-1}| > |A_0 \cup \ldots \cup A_{t-1} \cup E_0 \cup \ldots \cup E_{t-1}|/2=|G \setminus G_t|/2 > n^{1-\eps_n^2}/2.
\end{equation}
%Observe that the graphs are nested so $G=G_0 \supseteq
Observe that by our first win condition \Cref{cl:win-purepair-1} applied to the pure pair $(A_i,B_i)$, using property \ref{itm:b} we conclude
\begin{equation}\label{eq:Ai_upper_bound}
    |A_i| \le n^{1-16\eps_n^2}.
\end{equation}
In particular, this implies 
\begin{align}\label{eq:removed_upper_bound}
    |A_0 \cup \ldots \cup A_{t-1} \cup E_0 \cup \ldots \cup E_{t-1}| &\le |A_0 \cup \ldots \cup A_{t-2} \cup E_0 \cup \ldots \cup E_{t-2}|+2|A_{t-1}| \nonumber \\
    &\le |G \setminus G_{t-1}|+2|A_{t-1}| \le n^{1-\eps_n^2}+2n^{1-16\eps_n^2}      
\end{align}

A final, key property of the sets $A_i$, which will follow from the conditions we ensured as part of the main claim, is that any two of them make a pure pair. We prove this as the following claim.

\begin{claim}\label{cl:A_i-A_j-consistent}
For any $0\le i < j\le t-1$, we have that $(A_i,A_j)$ is a pure pair. 
\end{claim}
\begin{cla_proof}
Note that, since $V(G_j) \subseteq B_i \cup R_i$, the claim is equivalent to $A_j\subseteq B_i$ or $A_j \subseteq R_i$.
Suppose towards a contradiction that this is not the case. Then, we can find vertices $r\in A_j\cap R_i$ and $b\in A_j\cap B_i$. 

Consider $B_j\subseteq R_i\cup B_i$. Note that $r$ is adjacent in blue to $B_j$ (by property \ref{itm:b} since $r$ belongs to $A_j$) and that since it belongs to $R_i$, property \ref{itm:c} implies that it has blue degree at most $n^{\eps_n^2}|A_i|$ towards $B_i$. Hence, we can conclude that $|B_i \cap B_j|$ is at most of this size. This, in turn, implies  that
\begin{equation}\label{eq:final1}
|R_i \cap B_j| \ge |B_j|- n^{\eps_n^2}|A_i| \ge 2n^{1-\eps_n^2} -n^{1-15\eps_n^2} \ge 2n^{1-16\eps_n^2},
\end{equation}
where in the second inequality we used property \ref{itm:b} and \eqref{eq:Ai_upper_bound}.

We next wish to lower bound the size of $N:=R_j \cap B_i$. To this end, note that $|R_j \cap B_i|=|B_i|-|B_i \setminus R_j|$. Since $r \in A_j$, it is red to $R_j$ and blue to $B_j$. Note further that $R_j \cup B_j$ contains all vertices not in $A_0 \cup \ldots \cup A_j \cup E_0 \cup \ldots \cup E_j$, so $B_i \setminus R_j = B_i \cap (B_j \cup A_0 \cup \ldots \cup A_j \cup E_0 \cup \ldots \cup E_j)$. To upper bound this, note that  
$$|B_i \cap (A_0 \cup \ldots \cup A_j \cup E_0 \cup \ldots \cup E_j)| \le |A_0 \cup \ldots \cup A_j \cup E_0 \cup \ldots \cup E_j| \le n^{1-\eps_n^2}+2n^{1-16\eps_n^2},$$
by \eqref{eq:removed_upper_bound}. Note also that, since $r \in R_i$, by property \ref{itm:c} it has at most $n^{\eps_n^2}|A_i|$ blue neighbours in $B_i$. Since $r$ is blue to $B_j$, this implies that $|B_i \cap B_j|\le n^{\eps_n^2}|A_i|.$ Putting this together, we obtain 
\begin{align}
\label{eq:final2}
|R_j \cap B_i| = |B_i|-|B_i \setminus R_j| &\ge |B_i|- |B_i \cap (A_0 \cup \ldots \cup A_j \cup E_0 \cup \ldots \cup E_j)|-|B_i \cap B_j| \notag \\ &\ge |B_i| - n^{1-\eps_n^2}-2n^{1-16\eps_n^2} - n^{\eps_n^2}|A_i| \ge 3n^{1-16\eps_n^2},
\end{align}
where we used \eqref{eq:Ai_upper_bound} to bound $|A_i|$. This allows us to apply \Cref{lem:one-third-pure-pair} and through our first win condition (\Cref{cl:win-purepair-1})
to find some blue connected component $N'$ in $G[N]$ of size at least $|N|/3$. Note also that $N' \subseteq N \subseteq R_j$ since $N=R_j\cap B_i$.

Let now $b'\in R_i\cap B_j$ be arbitrary and suppose first that it is not consistent to $N'$. Then via \Cref{lem:cute-nice-obs} we can find vertices $u,v\in N'$ such that $uv$ is blue, $b'u$ is red, and $b'v$ blue. Note that $bb'$ is blue because $b'\in B_j$ and $b\in A_j$. We also know that $bv$ and $bu$ are red because $b\in A_j$ is red to $N = R_j\cap B_i \supseteq \{u,v\}$. Finally, pick an arbitrary $x\in A_i$ and observe that this ensures it is red to $b' \in R_i$ and blue to $b,u,v \in B_i$. Then, $xb'ubv$ induces a $P_5$ leading to a contradiction. See \Cref{fig:A_i-A_j-consistent} for an illustration.

\begin{figure}
\RawFloats
\begin{minipage}[t]{0.655\textwidth}
\captionsetup{width=\textwidth}
\centering

\begin{tikzpicture}[xscale=1.3]
\defPt{0}{0}{D}
\defPt{-2}{0}{C}
\defPt{2}{0}{E}
\defPt{0}{3}{A}
\defPt{3}{3}{B}

\foreach \i in {-2,...,2}
{
    \foreach \j in {-5,...,3}
    {
    \draw[line width = 0.5 pt,red] ($(A)+0.3*(\i,0)$) -- ($(C)+0.3*(-0.5*\j,\j)$);
    }
}

\foreach \i in {-2,...,2}
{
    \foreach \j in {-5,...,3}
    {
    \draw[line width = 0.5 pt,blue] ($(A)+0.3*(\i,0)$) -- ($(E)+0.3*(0.5*\j,\j)$);
    }
}

\fitellipsiss{$(A)+(0.9,-1.3)$}{$(A)-(0.9,1.3)$}{0.5};
\fitellipsiss{$(C)+(0.5,-0.8)$}{$(C)+(0.5,0.8)$}{1.3};
\fitellipsiss{$(E)+(-0.5,-0.8)$}{$(E)+(-0.5,0.8)$}{1.3};
\fitellipsisss{$(A)+(0.9,-1.3)$}{$(A)-(0.9,1.3)$}{0.5};
\fitellipsiss{$(A)+(0.6,0)$}{$(A)-(0.6,0)$}{0.5};
\fitellipsiss{$(B)+(-0.3,0)$}{$(B)-(0.9,0)$}{0.5};
\fitellipsiss{$(C)+(0.3,-1.75)$}{$(E)+(-0.3,-1.75)$}{0.5};

    \foreach \j in {-3,...,5}
    {
    \draw[line width = 0.5 pt,blue] ($(A)+(-0.9,-1.4)$) -- ($(C)+(0.5,0.5)+0.15*(\j,-\j)$);
    }

    \foreach \j in {-6,...,7}
    {
    \draw[line width = 0.5 pt,blue] ($(A)+(0.9,-1.4)$) -- ($(C)+0.15*(0,-\j)$);
    }

    \foreach \j in {-5,...,3}
    {
    \draw[line width = 0.5 pt,red] ($(A)+(0.9,-1.4)$) -- ($(E)-(0.5,-0.5)+0.15*(\j,\j)$);
    }

    \foreach \j in {-6,...,7}
    {
    \draw[line width = 0.5 pt,red] ($(A)+(-0.9,-1.4)$) -- ($(E)+0.15*(0,-\j)$);
    }

\fitellipsiss{$(C)+(0.5,-0.3)$}{$(C)+(0.5,0.3)$}{0.9};
\fitellipsiss{$(E)+(-0.5,-0.3)$}{$(E)+(-0.5,0.3)$}{0.9};
\fitellipsiss{$(E)+(-0.5,0.3)$}{$(E)+(-0.5,0.3)$}{0.7};

\pic[xscale=0.26,yscale=0.2,rotate=180,blue] at ($(E)+(-0.32,0.4)$) {conn};

\draw[rounded corners] (-3.4, -2.3) rectangle (3.8, 3.7) {};

\draw[] ($(A)+(0.9,-1.4)$) \vx;
\draw[] ($(A)+(-0.9,-1.4)$) \vx;

\draw[line width = 1 pt,red] ($(C)+(1,0)$) -- ($(E)+(-0.8,0.2)$);

\draw[line width = 1 pt,blue] ($(C)+(1,0)$) -- ($(E)+(-0.8,0.62)$);

\draw[] ($(C)+(1,0)$) \vx;
\draw[] ($(E)+(-0.8,0.2)$) \vx;
\draw[] ($(E)+(-0.8,0.62)$) \vx;
\draw[] ($(A)+(0,0)$) \vx;

\node[] at ($0.5*(C)+0.5*(E)+(0,-1.75)$) {\small $A_{i+1} \cup \ldots \cup A_{j-1} \cup E_{i+1} \cup \ldots  \cup E_{j}$};

\node[] at ($(-3.6,3.9)$) {$G_i$};

\node[] at ($(C)+(0.7,0)$) {\small $b'$}; 
\node[] at ($(A)+(1.15,-1.3)$) {\small $b$}; 
\node[] at ($(A)+(-1.15,-1.3)$) {\small $r$}; 
\node[] at ($(A)+(0,-1.3)$) {$A_j$};
\node[] at ($(A)+(0,0.3)$) {\small $x$};
\node[] at ($(A)+(-1.5,0)$) {$A_i$};
\node[] at ($(B)+(0.5,0)$) {$E_i$};
\node[] at ($(C)+(-1.1,0)$) {$R_i$};
\node[] at ($(C)+(0.5,-0.75)$) {\small $R_i \cap B_j$};
\node[] at ($(E)+(-0.5,-0.75)$) {\small $R_j \cap B_i$};
\node[] at ($(E)+(-0.05,0.2)$) {\small $N'$};

\node[] at ($(E)+(-0.67,0.8)$) {\small $v$}; 
\node[] at ($(E)+(-0.8,-0.1)$) {\small $u$};

\node[] at ($(E)+(1.1,0)$) {$B_i$};
\end{tikzpicture}
    \caption{Illustration of the structure in \Cref{cl:A_i-A_j-consistent}.}
    \label{fig:A_i-A_j-consistent}
\end{minipage}\hfill
\begin{minipage}[t]{0.33\textwidth}
\centering
\begin{tikzpicture}[scale=1.3]
\defPt{0.1}{-0.5}{x1}%b'
\defPt{1.2}{1.1}{x0}%x
\defPt{3}{-1.2}{x2}%u
\defPt{2.3}{0.4}{x3}%b
\defPt{3}{-0.6}{x4}%v

\fitellipsiss{$(x2)$}{$(x4)$}{0.5};
\fitellipsiss{$(x3)$}{$(x3)$}{0.5};
\fitellipsiss{$(x0)$}{$(x0)$}{0.5};
\fitellipsiss{$(x1)+(0,-0.3)$}{$(x1)+(0,0.3)$}{0.5};

\draw[line width = 2 pt, red]  (x4) -- (x3) -- (x2) -- (x1) -- (x0);

\draw[line width = 1 pt, blue] (x0) -- (x2);
\draw[line width = 1 pt, blue] (x0) -- (x4);
\draw[line width = 1 pt, blue] (x1) -- (x3);
\draw[line width = 1 pt, blue] (x3) -- (x0);
\draw[line width = 1 pt, blue] (x1) -- (x4);
\draw[line width = 1 pt, blue] (x4) -- (x2);

\node[] at ($(x1)+(-0.3,0)$) {$b'$};
\node[] at ($(x0)+(0,0.25)$) {$x$};
\node[] at ($(x2)+(0.15,-0.25)$) {$u$};
\node[] at ($(x3)+(0.2,0.25)$) {$b$};
\node[] at ($(x4)+(0.15,0.25)$) {$v$};
\node[] at ($0.5*(x2)+0.5*(x4)+(0.75,0)$) {$N$};
\node[] at ($(x0)+(0,0.75)$) {$A_i$};
\node[] at ($(x3)+(0.75,0)$) {$A_j$};
\node[] at ($(x1)+(0,-1.1)$) {$R_i \cap B_j$};

\foreach \i in {0,...,4}
{
\draw[] (x\i) \smvx;
}
\end{tikzpicture}
\captionsetup{width=\textwidth}
 \caption{$P_5$ found in \Cref{cl:A_i-A_j-consistent}}
    \label{fig:}
\end{minipage}
\end{figure}
Hence, any $b'\in R_i\cap B_j$ is consistent to $N'$. Keeping only the ones of majority colour gives us a pure pair with parts of size $|R_j \cap B_i|/2$ and $|N'|\ge |N|/3$ by \eqref{eq:final1} and \eqref{eq:final2}. This allows us to apply our first win condition, \Cref{cl:win-purepair-1}, one last time and obtain a contradiction.
\end{cla_proof}

Finally, note that by above claim and since $\log f(2^x)=cx^{2/3}$ is concave and increasing we may apply \Cref{r-partite-incons-bound} to the subgraph induced by $A_0, A_1, \ldots, A_t$ which consists of at least $s \ge \frac12 \cdot n^{1-\eps_n^2}$ vertices by \eqref{eq:leftover-small}. We also know by \ref{itm:a} and by \eqref{eq:Ai_upper_bound} that 
$$a:=n^{1-\eps_n} \le |A_i| \le n^{1-16\eps_n^2}:=b.$$
Note that this gives a contradiction unless $\eps_n \le 1/16$. With these parameters \Cref{r-partite-incons-bound} guarantees us a cograph of size at least either 
$$f\left(\frac{s}{2a \log (4b/a)}\right)\cdot f(a)\ge f\left(n^{\eps_n/2}\right)\cdot f(n^{1-\eps_n})\ge 2^{c/(2\eps_n^2)^{2/3}}  \cdot n^{c\eps_n-c\eps_n^2}\ge n^{c\eps_n^{5/3}/2}  \cdot n^{c\eps_n-c\eps_n^2}\ge n^{c\eps_n},$$
where in the second inequality, we used \eqref{eq:crit-G} for the first time in its stronger form (since we are applying it to a much smaller set than usual), and in the first inequality, we used 
$$\frac{s}{2a \log (4b/a)}\ge\frac{n^{1-\eps_n^2}/2}{2n^{1-\eps_n}\log (4n^{\eps_n-16\eps_n^2})}\ge \frac{n^{\eps_n-\eps_n^2}}{4\eps_n \log n} \ge n^{\eps_n/2}=2^{1/(2\eps_n^2)}.$$ 
Or \Cref{r-partite-incons-bound} gives us a cograph of size at least
$$f\left(\frac{s}{2b \log (4b/a)}\right)\cdot f(b)\ge f\left(n^{14\eps_n^2}\right)\cdot f(n^{1-16\eps_n^2}) \ge 2^{c( 14/\eps_n)^{2/3}}\cdot n^{c\eps_n-\eps_n^3}\ge n^{c\eps_n},$$ 
where in the last inequality we used $c(14/\eps_n)^{2/3}\ge 1$ and in the first inequality we used $$\frac{s}{2b \log (4b/a)}\ge \frac{n^{1-\eps_n^2}/2}{2n^{1-16\eps_n^2}\log (4n^{\eps_n-16\eps_n^2})}\ge \frac{n^{15\eps_n^2}}{4\eps_n \log n} \ge n^{14\eps_n^2}=2^{14/\eps_n}.$$ In either case, we found a cograph of size at least $f(n)$, which is a contradiction, completing the proof.
\end{proof}

\section{Concluding remarks}\label{sec:conc-remarks}
In this paper, we improved the best known bound on the size of a homogeneous set in a $P_5$-free graph from $2^{\Omega(\sqrt{\log n})}$ to $2^{\Omega((\log n)^{2/3})}$. We note that prior to this work, there has been no instance of the Erd\H{o}s-Hajnal conjecture in which such an intermediate result was known. The only instance in which any kind of intermediate result in this direction was known is the $C_5$-free case \cite{towards-5-hole}, which was quickly followed by a full resolution of that case in \cite{5-hole}.

While falling short of the full proof of the conjecture, we do believe several of our ideas could be very useful in resolving the conjecture in full. In fact, we deliberately wrote the proof in the form of trying to find a cograph of size $n^{\eps_n}$ (where $\eps_n= \frac{1}{(\log n)^{1/3}}$) to showcase that our arguments do apply for quite some time and already establish quite a lot of structure even if one only works with $\eps_n$ being a constant. Namely, the current bottleneck for our argument is our first ``win'' condition (\Cref{cl:win-purepair-1}), which states that we find a large enough cograph (so also a homogenous set) provided we find a pure pair with both parts of size at least $n^{1-O(\eps_n^2)}$. On the other hand, our second ``win'' condition (\Cref{cl:win-purepair-2}) which states that we find a large enough cograph if we can find a pure pair with one part of size at least $n^{1-O(\eps_n)}$ and the other $n-n^{1-\Omega(\eps_n^2)}$ applies even if $\eps_n$ is a constant. In fact, by considering a certain very unbalanced blow-up of $C_5$, where we place optimal $P_5$-free graphs inside of each part and choose the largest part to have size $n(1-o(1))$ shows that one may be forced to consider such extremely unbalanced pairs. On the other hand, our argument is precisely geared towards showing critical examples actually need to resemble this construction. %, see e.g.\ \Cref{fig:overall} and note that if $G[B]$ were connected in blue then $C$ and $E$ would also be joined by red edges giving a blow-up of $C_5$ in red (minus a set $X$ the size of which is small enough compared to $|A|$ that it may be safely ignored). 

We note that while in this paper we focused on the case of $P_5$-free graphs, a number of our ideas do apply when working with more general forbidden graphs. That said, a number of ideas that lead to breaking the $2^{\Omega(\sqrt{\log n \log \log n})}$ barrier are fairly specific to $P_5$. However, as we discussed in the introduction already, matching this barrier is a highly non-trivial task despite only being a modest-seeming improvement and we would find a result establishing it for all graphs to be a very interesting, natural and perhaps a bit more approachable step towards proving the Erd\H{o}s-Hajnal Conjecture in full. 
\begin{prob}
Given any graph $H$ show that for any $n$-vertex $H$-free graph $G$ we have  $\hom(G) \ge 2^{\Omega(\sqrt{\log n \log \log n})}$.
\end{prob}
It is not too hard to extend our ideas slightly and to prove such a result for an infinite family of prime graphs (not obtainable by substitution). Since, in this paper, the focus was on breaking the barrier, we left this direction for future work. 
This question has subsequently been resolved by the second author, Nguyen, Scott and Seymour \cite{tung-paper}. 

In terms of breaking the barrier, the $P_5$-free case is the only one in which we have such an intermediate result towards the Erd\H{o}s-Hajnal conjecture. Here, the case of longer paths seems like the most natural next step since a number of our ideas in this direction could be useful to extend to longer paths.  

\begin{prob}\label{prob:long-path}
Let $k$ be a positive integer. Show that there exists an $\eps>0$ such that any $P_k$-free $n$-vertex graph $G$ has $\hom(G) \ge 2^{\Omega((\log n)^{1/2+\eps})}.$
\end{prob}

\textbf{Note added in proof.} We note that while this paper was under review \Cref{prob:long-path} has been resolved in a spectacular fashion by Nguyen, Scott and Seymour \cite{long-paths} who show a much stronger bound of $2^{(\log n)^{1-o(1)}}$ which in particular also improves our main result \Cref{thm:main}. 

\textbf{Acknowledgements.} We would like to thank Noga Alon, Maria Chudnovsky, Jacob Fox, Matthew Kwan, Tung Nguyen, Alex Scott and Paul Seymour for their useful comments. We are also very grateful to the anonymous referees for their very careful reading, plenty of suggestions on how to improve the manuscript, and minor corrections. The first author would like to thank the Princeton University Department of Mathematics for the support provided through the Summer Undergraduate Research Program 2022. The second author would like to gratefully acknowledge the support of the Oswald Veblen Fund.

%\textbf{Data availability.} This project involved no data.

%\textbf{Conflict of interest.} The authors have no conflicts of interest related to the work in this paper.

\providecommand{\MR}[1]{}
\providecommand{\MRhref}[2]{%
  \href{http://www.ams.org/mathscinet-getitem?mr=#1}{#2}
}

%  \bibliographystyle{amsplain_initials_nobysame}
%  \bibliography{ref}

\providecommand{\bysame}{\leavevmode\hbox to3em{\hrulefill}\thinspace}
\providecommand{\MR}{\relax\ifhmode\unskip\space\fi MR }
\providecommand{\MRhref}[2]{%
  \href{http://www.ams.org/mathscinet-getitem?mr=#1}{#2}
}
\providecommand{\href}[2]{#2}

\appendix
\section{Substitution for subpolynomial Erd\H{o}s-Hajnal bounds}\label{sec:substitution}
Given graphs $H, F$ and a vertex $v$ of $H$ we denote by $H_v(F)$ the graph obtained by replacing $v$ with an induced copy of $F$ and joining it to the remaining vertices of $H$ in the same manner as $v$. This operation is called \emph{substitution} and a beautiful result of Alon, Pach and Solymosi \cite{APS01} states that the family of graphs for which the Erd\H{o}s-Hajnal conjecture holds is closed under substitution, in the sense that if it holds for graphs $H, F$, then for any vertex $v$ of $H$, it also holds for $H_v(F)$. In this section, we will show, for the sake of completeness and to provide a reference, that as long as a function $f$ is ``well-behaved'', the same result holds for the family of graphs $H$ with the property that there exists an $\eps=\eps(H,f) >0$ such that any $n$-vertex $H$-free graph has a homogenous set of size at least $(f(n))^{\eps}$. We will refer to a graph $H$ with such a property as an \emph{$f$-Erd\H{o}s-Hajnal} graph (regardless of the properties of $f$). We note that variants of this statement are well-known to hold in the area and that the formulation we use here is based on a variant that appeared in an early draft of \cite{tung-paper}. 

Let us first formalise what we mean by well-behaved above. A function $f: [1,\infty) \to \mathbb{R}^+$ is said to be \emph{almost super multiplicative} if there exists an $1\ge \eps>0$ such that for all $x,y \in [1,\infty)$ we have $f(x)f(y) \ge (f(xy))^{\eps}.$ For example, setting $\eps=1$ precisely recovers super multiplicativity, while $f(x)=2^{(\log x)^{2/3}}$ is almost super multiplicative with $\eps=\frac12$ since 
$$2^{(\log x)^{2/3}+(\log y)^{2/3}} \ge 2^{\frac12 (\log x+\log y)^{2/3}}\ge 2^{\frac12(\log(xy))^{2/3}}.$$ 

The following lemma extends the result of Alon, Pach and Solymosi \cite{APS01} and follows by essentially the same proof. It says that for any almost super multiplicative function $f$, the class of $f$-Erd\H{o}s-Hajnal graphs is closed under substitution. 
\begin{lem}[\cite{tung-paper}]\label{subpoly-E-H}
Let $f$ be an almost super multiplicative increasing function and $H, F$ be $f$-Erd\H{o}s-Hajnal graphs. Then, for any $v \in V(H)$ the graph $H_v(F)$ is also $f$-Erd\H{o}s-Hajnal. 
\end{lem}
\begin{proof}
If $|H|=1$, the theorem is immediate, so let us assume $|H|\ge 2$.
Let $\eps_H:=\eps(H,f)$ and $\eps_F=\eps(F,f)$, and $\eps_f$ be the value of $\eps$ with which $f$ is almost super multiplicative. 

Let $\eps=\min \left \{\frac{\eps_H(\eps_f)^{4|H|}}{4|H|},\frac{\eps_F\eps_f}{2}\right\}.$ We will show that any $n$-vertex graph $G$ has $m:=\hom(G)\ge (f(n))^{\eps}.$ This is trivially satisfied if $f(n)<1$ so let us assume $f(n) \ge 1$. Let $s$ be the minimum number with the property that any $s$ vertices of $G$ contain a copy of $H$. Since $s\ge |H|\ge 2$, we can only have $s = 2$ if $G$ is a homogenous set itself and $m=n \ge (f(n))^{\eps_H}\ge (f(n))^{\eps},$ where in the first inequality we used the assumption that $H$ is $f$-Erd\H{o}s-Hajnal. So we may assume $s \ge 3$.

Note further that by $H$ being $f$-Erd\H{o}s-Hajnal, and the definition of $s$ we know $(f(s-1))^{\eps_H}\le m.$ Now if $s > {n}^{\frac1{4|H|}}+1,$ then we get  
$$m\ge (f(s-1))^{\eps_H}> \left(f\left({n}^{1/{(4|H|)}}\right)\right)^{\eps_H}\ge \left(f\left({n}\right)\right)^{{\eps_H(\eps_f)^{4|H|}}/{(4|H|)}}\ge (f(n))^{\eps},$$
as desired. Here, in the first inequality, we used that $f$ is increasing, and in the second, we used the almost super multiplicativity of $f$. So we are left with the case when $3\le s \le {n}^{1/{(4|H|)}}+1.$ This implies ${n}^{1/{(4|H|)}} \ge 2$, which in turn gives $s\le {n}^{1/{(4|H|)}}+1 \le {n}^{1/{(2|H|)}}.$

On the other hand, we know that there are more than $$\frac{\binom{n}{s}}{\binom{n-|H|}{s-|H|}}> \left(\frac{n}{s}\right)^{|H|}$$ copies of $H$ in $G$. 
This implies there is a copy of $H \setminus \{v\}$ in $G$ which has at least ${n}/{s^{|H|}}$ extensions into $H$. In other words, there exists $H' \subseteq G$, isomorphic to $H \setminus \{v\}$ and a set $S \subseteq G$ of at least ${n}/{s^{|H|}}$ vertices, with the property that for every $u \in V(H')$ if its corresponding vertex was adjacent to $v$ in $H$ then $u$ is complete to $S$ and is anticomplete otherwise. In particular, for any $w \in S$ we have $V(H') \cup w$ induce a copy of $H$. This implies that $G[S]$ is $F$-free, as any copy of $F$ inside of $S$ would give rise to a copy of $H_v(F)$ in $G$. Since $F$ is $f$-Erd\H{o}s-Hajnal, this implies 
$$m \ge (f(|S|))^{\eps_F}> \left(f\left(\frac{n}{s^{|H|}}\right)\right)^{\eps_F} \ge \left(f\left(\sqrt{n}\right)\right)^{\eps_F} \ge (f(n))^{\eps_F \eps_f/2}\ge (f(n))^{\eps}$$
where we used that $f$ is increasing in both the second and third inequality, where in the third inequality, we also used the above upper bound on $s$, and we used almost super multiplicativity in the penultimate inequality. 
\end{proof}

An immediate corollary of this result combined with \Cref{thm:main} is that, given any graph $H$ which can be obtained by iterative substitution using in each step $P_5$ or any Erd\H{o}s-Hajnal graph, then any $n$-vertex $H$-free graph $G$ has $\hom(G) \ge 2^{\Omega((\log n)^{2/3})}.$
\end{document}